\documentclass[journal]{IEEEtran}
\usepackage[T1]{fontenc}
\usepackage{times}
\usepackage{amsmath,amsfonts,amssymb}
\usepackage{amsthm}
\usepackage{bbm}
\usepackage{algorithmic}
\usepackage{algorithm}
\usepackage{array}
\usepackage[caption=false,font=normalsize,labelfont=sf,textfont=sf]{subfig}
\usepackage{textcomp}
\usepackage{stfloats}
\usepackage{url}
\usepackage{verbatim}
\usepackage{graphicx}
\usepackage{cite}
\usepackage{hyperref}
\usepackage[table]{xcolor}
\usepackage{float}
\usepackage{tabularx}
\usepackage{multirow}
\usepackage{bbold}
\usepackage{dsfont}
\usepackage{enumitem}
\usepackage{booktabs}
\usepackage{xparse}
\usepackage{cleveref}
\usepackage{catchfilebetweentags}
\usepackage{cite}

\newcolumntype{Y}{>{\centering\arraybackslash}X} 

\DeclareMathOperator*{\argmax}{arg\,max}
\DeclareMathOperator*{\argmin}{arg\,min}
\renewcommand{\algorithmicrequire}{\textbf{Input:}}
\renewcommand{\algorithmicensure}{\textbf{Output:}}

\newtheorem{lemma}{Lemma}

\newtheorem{proposition}{Proposition}

\newcommand{\cc}{\mathrm{c}}
\newcommand{\dd}{\mathrm{d}}
\newcommand{\g}{\mathrm{g}}
\newcommand{\cJ}{\mathcal{J}}
\newcommand{\cO}{\mathcal{O}}
\newcommand{\R}{\mathbbm{R}}
\newcommand{\w}{\overline{w}}

\newcommand{\cY}{\mathcal{Y}}
\newcommand{\cZ}{\mathcal{Z}}

\newcommand{\br}{\mathrm{br}}
\newcommand{\brc}{\mathrm{brc}}
\newcommand{\brs}{{\mathrm{br}*}}
\newcommand{\nbr}
{\mathrm{non\text{-}br}}

\newcommand{\dc}
{\mathrm{dc}}
\newcommand{\bund}{\mathrm{BUND}}
\newcommand{\chg}{\mathrm{chg}}
\newcommand{\dis}{\mathrm{dis}}
\newcommand{\tem}{\mathrm{em}}
\newcommand{\es}{\mathrm{es}}
\newcommand{\esp}{\mathrm{es\text{-}p}}
\newcommand{\ese}{\mathrm{es\text{-}e}}
\newcommand{\uc}{\mathrm{uc}}
\newcommand{\hy}{\mathrm{hy}}
\newcommand{\full}{\mathrm{full}}
\newcommand{\tni}{\mathrm{ni}}
\newcommand{\tnis}{{\mathrm{ni}*}}

\newcommand{\sh}{\mathrm{sh}}
\newcommand{\vio}{\mathrm{vio}}

\newif\ifshowedits
\showeditsfalse  

\ifshowedits
  \NewDocumentCommand{\edit}{+m}
    {{\color{blue}#1}}
  \newcommand{\rev}[1]{\textcolor{blue}{#1}}
\else
  \NewDocumentCommand{\edit}{+m}
    {#1}
  \newcommand{\rev}[1]{#1}
\fi

\begin{document}

\title{Contingency-Aware Nodal Optimal Power Investments with High Temporal Resolution}

\author{Thomas Lee, \IEEEmembership{Student Member, IEEE}, Andy Sun, \IEEEmembership{Senior Member, IEEE.}

\thanks{Supported by MIT Energy Initiative's Future Energy Systems Center.}
\thanks{Thomas Lee is with the Institute for Data, Systems, and Society, MIT.}
\thanks{Andy Sun is with the Sloan School of Management, MIT.}}

\markboth{}%
{}



\maketitle

\begin{abstract}
We present CANOPI, a novel algorithmic framework, for solving the Contingency-Aware Nodal Power Investments problem, a large-scale nonlinear optimization problem that jointly optimizes investments in generation, storage, and transmission upgrades, including representations of unit commitment and long-duration storage. \rev{T}he problem's scale arises from the confluence of spatial and temporal resolutions\rev{, along with the large number of contingency constraints. Further, t}he underlying problem is nonlinear due to transmission upgrades' \rev{impact} on impedances. We propose algorithmic approaches to address these computational challenges. We pose a linear approximation and develop a fixed-point algorithm to adjust for nonlinear impedance feedback. We solve the large-scale linear expansion model with a specialized level-bundle method leveraging a novel interleaved approach to contingency constraint generation. We introduce a minim\rev{um} cycle basis algorithm that improves numerical sparsity \rev{and solve times} of cycle-based DC power flow. CANOPI is demonstrated on a 1493-bus Western Interconnection test system built from realistic-geography network data, with hourly operations spanning 52 week-long scenarios and a total possible set of 20 billion individual transmission contingency constraints. Numerical \rev{experiments} quantify reliability and economic benefits of incorporating transmission contingencies in integrated planning models and highlight the computational advantages of the proposed methods.

\end{abstract}


\section{Introduction}

Capacity expansion models are crucial tools for grid planners, regulators, and utilities to systematically plan long-lived electricity infrastructure, including generation, storage, and transmission, while representing physical, engineering, and policy constraints. The core computational challenge lies in the coupling between long-term investment decisions and short-term operational constraints over multiple time periods. Detailed time domain resolution is required to represent vital clean energy technologies: meteorology drives temporal variation in wind and solar availability, and the value of energy storage emerges from operation over consecutive time periods.

Spatial coupling is also critical. The abundance of wind, solar, and geothermal depends on siting. Generation and storage have strong interactions with the transmission network \cite{krishnan2016co}, e.g. via substitution effects. Consequently, network spatial resolution significantly impacts the accuracy of capacity expansion models
\cite{frysztacki2021strong, jacobson2024quantifying, serpe2025importance}. Moreover, power systems in the US face accelerating load growth, driven by factors including AI data centers and electrification, concurrently with historically slow transmission expansion \cite{doe2023transmission}. Two related critical transmission bottlenecks remain: resource \emph{interconnection} (stranding terawatts of generation and storage projects \cite{rand2024queued}) and transmission \emph{congestion} during grid operations (raising costs and causing renewable energy curtailment).


\rev{The interaction of} major system changes with these two grid phenomena (interconnection and congestion) \rev{depends on} the security-constrained power flow models that underlie power transmission physics. The vast majority of interconnection and congestion constraints are driven by transmission \emph{contingencies}, which are NERC-enforced constraints to withstand the loss of any single transmission component \cite{nerc2023tpl001}. \rev{In a power flow model, each transmission constraint is either a ``base-case'' (no contingency) or contingency constraint.} Table~\ref{tab:caiso_congestion} shows the \rev{high proportional} importance of contingencies \rev{in real-life grids.}\footnote{{Sources: For interconnection planning (data available only for PJM), we show the share of contingency-caused binding constraints in sampled projects' System Impact Studies. Day-ahead congestion results are shown for July 2024.}} This suggests base-case power flows are insufficient.


\vspace{-10pt}
\begin{table}[ht]
\centering
\caption{Transmission constraint causes: contingency vs. base-case.}
\vspace{-5pt}
\label{tab:caiso_congestion}
\begin{tabular*}{\columnwidth}{@{\extracolsep{\fill}}  c c c  }
\toprule
\textbf{Grid process} & \textbf{Region} &  \textbf{\rev{\% caused by contingency}}
\\
\midrule
{{Interconnection planning}} & PJM & 90\% 
\\
\hline
  &      PJM & 86\%  
\\
{Day-ahead market}  & ERCOT & 98\%  
\\
  & CAISO & 93\%  
\\
\bottomrule
\end{tabular*}
\end{table}
\vspace{-5pt} 

Co-optimizing transmission and generation can significantly improve efficiency compared to decoupled planning \cite{krishnan2016co}. However, a lack of holistic planning tools has forced these two sides to remain largely separate processes requiring manual iterations. In fact, inefficient coordination between generation and transmission planning is a major cause for grid underinvestment and queue delays ({e.g., developers' multi-site speculative interconnection requests}) \cite{mays2023generator}. This motivates {expansion} models with high temporal \emph{and} spatial resolutions.

\subsection{Literature review}

Prior studies address generation and transmission expansion with varying levels of detail. \rev{Some past works on} coordinated generation-transmission planning \rev{employ detailed power flow models, which could explicitly model contingencies, but they are} limited in network size \cite{roh2009market, motamedi2010transmission} or temporal scope \cite{zhang2018mixed, mehrtash2020security,degleris2024gpu}. \rev{Modern} zonal models \cite{jacobson2024quantifying, serpe2025importance, pecci2025regularized} \rev{can cover large regions and detailed time resolution, and their zonal transfer limits could represent contingencies in an aggregated manner (if derived from underlying nodal networks)} \cite{brown2023general, sergi2024transmission}. Yet transfer capability fluctuates with system conditions \rev{instead of staying at a pre-computed} value \cite{nerc2024itcs} especially under evolving resource mixes or rapid load growth, which are precisely the situations studied by capacity expansion models. For example, 13 of PJM’s top 25 transmission constraints in 2023 did not appear in the 2022 list \cite{monitoring2023state}, due to \rev{system changes}.

\rev{Several works (e.g., the PyPSA model) heuristically derate lines to 70\% of nominal capacity} \cite{horsch2018linear, neumann2019heuristics, frysztacki2021strong}, \rev{but this has no theoretical guarantee and can lead to either over- or underinvestment. This paper will empirically test this heuristic. Another approach} \cite{DOE_NTP_2024} sequentially solve\rev{s} a zonal capacity expansion model, and then downscale\rev{s} the results for nodal power flow simulations \rev{(which can model explicit contingencies)}. \rev{This approach also lacks} optimality guarantee\rev{s} and is highly sensitive to the downscaling heuristics employed.





\vspace{-6pt}
\begin{table}[ht]
    \centering
    \caption{Comparison of prior literature vs. this work.}
\vspace{-5pt}
    \begin{tabular*}{\columnwidth}{@{\extracolsep{\fill}}  c  c  c  }
    \toprule
       \textbf{Papers}  &   \textbf{Limitations}  &  \textbf{This paper}
\\
\midrule
       \cite{roh2009market, motamedi2010transmission} &  Small network (5{$\sim$}6 buses) & 1,493 buses
\\
    \cite{zhang2018mixed, mehrtash2020security,degleris2024gpu} & Small timescale (1{$\sim$}192 hours) & 8,736 hours
\\
\hline
   \cite{jacobson2024quantifying, serpe2025importance, pecci2025regularized} & \rev{Security-informed zonal limits} & \rev{Explicit $n{-1}$ security}
\\
    \cite{horsch2018linear, neumann2019heuristics, frysztacki2021strong}
    & \rev{Heuristic scalar security factor} & \rev{Explicit $n{-}1$ security}
\\
    \cite{DOE_NTP_2024} & Sequential zonal CEM $\to$ nodal & \emph{Integrated} nodal CEM
\\
\bottomrule
    \end{tabular*}
    \label{tab:lit_review}
\end{table}
\vspace{-7pt}




\subsection{Contributions}

In this paper, we develop a novel algorithmic framework to co-optimize generation-storage-transmission in a holistic manner with high temporal resolution, while \rev{explicitly modeling} $n{-}1$ security-constrained nodal DC power flows \rev{(SC-DC)}. To achieve this, we make the following contributions:
\begin{enumerate}[leftmargin=*]
    \item We formulate a capacity expansion model with \rev{SC-DC constraints}, while capturing a nonlinear effect of impedance\rev{-capacity} feedback. To enable tractability, we \rev{linearly approximate the model} and develop a fixed-point correction algorithm to reconcile the nonlinear relationship.

    \item To solve the linear approximation, we design a novel variant of the level-bundle method that combines analytic-center stabilization with \emph{interleaved} contingency constraint generation. This tractable algorithm is scalable to cover billions of potential transmission contingency constraints.

    \item {We introduce the usage of \emph{minim\rev{um}} cycle bases for DC power flow constraints, and develop a novel integer programming algorithm to efficiently compute these bases}.

    \item We \rev{quantify the impacts of explicit} nodal contingencies \rev{through controlled numerical experiments, using} our algorithmic \rev{framework} (named CANOPI) \rev{on} a realistic Western Interconnection network with detailed hourly operations.
    
\end{enumerate}

\section{Model}
\label{sec:model}

In this section, we introduce our formulation of the capacity expansion problem. Assume the connected network has $N$ nodes and $b$ AC transmission branches with the branch incidence matrix $A^\br \in \{-1,0,1\}^{N\times b}$. Each branch $j$ has arbitrarily assigned ``from'' and ``to'' buses $i^{fr}_j, i^{to}_j$ with entries $A^\br[i^{fr}_j,j]={1}$ and $A^\br[i^{to}_j,j]={-}1$. There are $\beta$ HVDC lines with an incidence matrix $A^\dc \in \{-1,0,1\}^{N\times \beta}$. {This work models one year of operations, divided into $|\Omega| = 52$ \rev{sequential} weeks as scenarios indexed by $\omega \rev{\,=1,...,52}$}.

\rev{This framework can also be extended to multiple weather years, or additional scenarios for high-stress weeks or days.}

\subsection{Approach to time-coupling constraints}
\label{sec:time_coupling}
\rev{Long-duration energy storage (LDES) energy states are} linked between adjacent weekly scenarios $\omega$ and $\omega{+}1$. \rev{Following} \cite{pecci2025regularized}, generator ramping, short-duration storage, and unit commitment \rev{(UC)} are modeled with intra-week ``cyclic constraints'', \rev{where} each week\rev{'s} first hour is modeled as being immediately after the same week\rev{'s} last hour. \rev{As in} \cite{pecci2025regularized}, \rev{LDES energy states are clustered by zone, and UC constraints use a continuously-relaxed clustering} by zone and technology \rev{(e.g., coal, combined-cycle, combustion turbine)}. This combination of time-coupling assumptions \rev{(summarized in Table \ref{tab:temporal_coupling}) is} discussed in detail by \cite{pecci2025regularized}, including limitations and extensions. \rev{Multistage operational uncertainty (e.g., \cite{lorca2016multistage,bodal2022capacity}) is outside the scope of this work.}

\vspace{-7pt}
\begin{table}[H]
    \caption{Overview of time-coupling and spatial assumptions}
\vspace{-5pt}
    \centering
    \begin{tabular}{c c c  c c}
    \toprule
      \textbf{Model aspect}  & \textbf{Time-coupling} & \cite{pecci2025regularized} & \textbf{This paper}
      \\
      \hline
      Generator ramping & Intra-week cyclic & Zonal & \emph{Nodal}
      \\
      Short-duration storage & Intra-week cyclic & Zonal & \emph{Nodal}   
      \\
      Unit commitment & Intra-week cyclic & Zonal & Zonal
      \\
      \rev{LDES energy states} & Multi-week linked & Zonal & Zonal
      \\
    \bottomrule
    \end{tabular}
    \label{tab:temporal_coupling}
\end{table}
\vspace{-5pt}

\rev{To justify these modeling choices, we quantify the effect of changing \emph{intra-week} to \emph{full-year} constraints to be ${\sim}0.1\%$ of total costs (in our test case). Next, we quantify the effect of \emph{zonal vs. nodal LDES} states to be ${\sim}$0.2\% of total costs (in our test case). Based on prior literature \cite{palmintier2013heterogeneous,poncelet2020unit}, we estimate that \emph{continuous clustering for UC} affects total costs by ${\sim}$0.4\%.\footnote{\rev{In \cite{palmintier2013heterogeneous}, individual-unit UC MILP vs. technology-clustered MILP differ in fuel costs by $0.7\%$. In \cite{poncelet2020unit}, clustered MILP vs. clustered LP differ in total costs by $0.2\%$. Thus, UC \emph{continuous clustering} impacts total costs by roughly $0.4\% \approx 0.7\% \times 24\% + 0.2\%$ (since ${\sim}24\%$ of our costs are fuel costs).}}}

While this paper focuses \rev{only} on existing reservoir-based hydropower, this framework can be extended to \rev{other LDES} technologies. \rev{Compared} to \cite{pecci2025regularized}, our model adds nodal resolution for ramping and short-duration storage, constituting a Pareto improvement \rev{for the time-coupling constraints}.




{
\subsection{Focus on fixed-topology, \rev{continuous} branch upgrades}
\label{sec:fixed_topology}
``Speed to power'' \rev{is now} a primary motivation in modern power grids amid rapid load growth. \rev{While} greenfield transmission faces \rev{permitting challenges and stakeholder opposition (often taking} 5--15 years), brownfield reconductoring along existing corridors can be completed in 1.5--3 years \rev{at ${\sim}$}2x \rev{lower} cost \cite{gridlab2024reconductoring}. In \cite{chojkiewicz2024accelerating}, a nationwide \rev{study} suggests reconductoring can contribute over 80\% of new transmission capacity by 2035. Motivated by speed, affordability, and relevance, this work focuses on brownfield fixed-topology branch upgrades.


CANOPI features continuous transmission sizing and impedance feedback correction that \rev{approximates} the full menu of granular conductor options \rev{(e.g.,} Southwire offers 70 ACSR sizes differ\rev{ing} by only 5MW on average).\footnote{{Southwire. ``Overhead Conductor Manual 2nd Ed.'' 5MW assumes 230kV.}} \rev{For} transmission expansion planning (TEP) with \emph{new} corridors, one could pre-identify candidate corridors and then use our framework to co-optimize transmission sizing. In contrast to TEP models with one binary variable \rev{per line, the continuous model could lead to more mathematically efficient outcomes.}


\rev{When topology must be optimized endogenously, or when transmission owners are restricted to a small set of standard cable types (e.g., to simplify spare parts or maintenance), this paper's exclusion of integer branch constraints is a limitation. Nonetheless, we remark that the continuous relaxation is a \emph{necessary} step in exact MILP solvers (e.g., branch-and-bound), so one can view this paper as a necessary milestone to integer constraints.} \rev{We comment further on this point in Section \ref{sec:extension}.}

\subsection{Capacity expansion problem}
\label{sec:cem}

We \rev{optimize} {first-stage decisions} $x = (x^\g, \allowbreak x^\es, \allowbreak x^\br, \allowbreak {x^\mathrm{inv}}, \allowbreak x^\mathrm{em})$ consisting of \emph{new} generation capacities $x^\g \,{\in}\, \R^{n^\g}$ of $n^\g$ generators, capacities of power $x^\esp \, {\in} \, \R^{n^\es}$ and energy $x^\ese \, {\in}\, \R^{n^\es}$ of {$n^\es$ short-duration} storage devices, capacities $x^\br \,{\in}\, \R^b$ of $b$ AC transmission branches, \rev{weekly} states $x^\mathrm{inv} \,{\in}\, \R^{ |\Omega|\cdot |\mathcal{G}^\hy|}$ for $ |\mathcal{G}^\hy|$ \rev{LDES} clusters, and the allocation of a policy metric $x^\tem \,{\in}\, \R^{|\Omega|}$ across \rev{weeks}. The variables $x^\mathrm{inv}_{\omega}$ \rev{are LDES} states at the beginning of \rev{each week} $\omega$.

\rev{The mathematical technique of using first-stage ``budgeting variables'' to enforce year-long coupling constraints (e.g., annual limits on emissions or fossil generation) is introduced by \cite{jacobson2024computationally} and extended to multi-week LDES state trajectories by \cite{pecci2025regularized}. We follow this approach.}


We model transmission capacity upgrade along existing branches as a continuous variable $x^\br$, \rev{as discussed above}. HVDC upgrades can \rev{also} be easily incorporated. In this setting, the incidence matrices $A^\br, A^\dc$ remain constant. With upper bound $\overline{x}$, i.e., limits on investments and reservoir \rev{states}, and a policy \rev{goal} $\overline{x}^\tem$, the feasibility region of $x$ is a polytope
\begin{equation}
\label{eq:investment_constraints}
    \mathcal{X} = \left\{x: 0\leq x \leq \overline{x},\ \mathbf{1}^\top x^\tem \leq \overline{x}^\tem \right\}.
\end{equation}


We define the overall capacity expansion model (CEM) as
\begin{align}
(\text{{CEM}}) \ \min_{x \in \mathcal{X} } \quad &c^\top x + \sum\nolimits_{\omega \in \Omega} h(x,\xi_\omega),\label{eq:cem}
\end{align}
where $c^\top x$ is the capacity investment cost and $h(x,\xi_\omega)$ is the optimal cost of the operational subproblem with new capacity portfolio $x$ in scenario $\xi_\omega$. The details of the scenarios $\xi_\omega$ and the value function $h$  will be given in Section \ref{subsec:op_constraints}. 

\subsection{Detailed operational problem}
\label{subsec:op_constraints}

An operational scenario $\omega$ is defined by its data vector $\xi_\omega = [c^\g_\omega, a^\g_\omega, \overline{p}^\mathrm{d}_\omega]$ of the generator operating costs $c^\g_\omega\in\R^{T n^\g}$ over $T$ hours, the generators' hourly availability factors $a^\g_\omega\in\R^{T n^\g}$, and load levels $\overline{p}^\dd_\omega\in\R^{T {n^\dd}}$ of {$n^\dd$} loads. Each scenario $\omega$ has an operational subproblem, introduced below.


\subsubsection{Generation and Short-Duration Storage}\label{sec:gen_stor}
Generators satisfy the following standard operational constraints,
\begin{subequations}
    \label{eq:gen_limits}
    \begin{align}
        &p^\g_{\omega t} + r^\g_{\omega t} \leq a^\g_{\omega t} \odot (\w^\g + x^\g), \ \forall t\in[T], \label{eq:gen_power_reserves}
\\
        &p^\g_{\omega, t+1} - p^\g_{\omega t} \geq -R\odot(\w^\g + x^\g),\ \forall t\in[T], \label{eq:gen_rampdown}
\\
        &p^\g_{\omega, t+1} - p^\g_{\omega t} \leq R\odot(\w^\g + x^\g), \ \forall t\in[T], \label{eq:gen_rampup}
    \\
        &\sum\nolimits_{t\in[T]} e^\top p_{\omega t}^g \leq x^\mathrm{em}_\omega, \label{eq:gen_em}
    \\
        & p^\g_{\omega t}, r^\g_{\omega t} \geq 0, \ \forall t\in[T], \label{eq:gen_nonneg}
    \end{align}
\end{subequations}
where $p^\g_{\omega t}, r^\g_{\omega t}\in\R^{n^\g}$ are vectors of power generation and reserves at time $t$, and $\w^\g \in \R^{n^\g}$ are existing generator capacities. Eq. \eqref{eq:gen_power_reserves} limits the power output and reserve of each generator to its physical availability $a^g_{\omega t}$, where $\odot$ denotes element-wise product. Eq. \eqref{eq:gen_rampdown}-\eqref{eq:gen_rampup} enforce ramp-down and ramp-up limits, respectively, with the ramp rate vector $R$. Note that in \eqref{eq:gen_rampdown}-\eqref{eq:gen_rampup}, when $t=T$, {we let $p^\g_{\omega, T+1}$ stand for $p^\g_{\omega 1}$, creating the cyclic ``wrapping'' constraints described in Section \ref{sec:time_coupling}}. Eq. \eqref{eq:gen_em} uses an emissions factor vector $e \in \R^{n^\g}$ to limit total fossil generation to the allocated budget $x^\mathrm{em}_\omega$.

\edit{
\begin{lemma}[Year-long policy coupling]
    Problem \eqref{eq:cem} and Eq. \eqref{eq:gen_em} enforce the full-year policy constraint
    $$\sum\nolimits_{\omega\in\Omega} \sum\nolimits_{t\in[T]} e^\top p_{\omega t}^g \leq \overline{x}^\mathrm{em}.$$
\end{lemma}
\begin{proof}
This results directly from summing \eqref{eq:gen_em} over all $\omega \in \Omega$, and substituting $\sum_{\omega \in \Omega} \sum\nolimits_{t\in[T]} e^\top p_{\omega t}^g \leq \sum_{\omega\in\Omega} x^\mathrm{em}_\omega$ into problem \eqref{eq:cem}'s constraint $\sum_{\omega\in\Omega} x^\tem_\omega \leq \overline{x}^\tem$.
\end{proof}
}

Storage devices face power and energy constraints
\begin{subequations}
\label{eq:storage}
\begin{align}
    & p^{es}_{\omega t} = p_{\omega t}^\text{dis} - p_{\omega t}^\text{chg}, \ \forall t \in [T],  
\\
    & p^\mathrm{chg}_{\omega t} + p^\mathrm{dis}_{\omega t} + r^\text{dis}_{\omega t} \leq  \w^\esp + x^\esp, \quad \forall t \in [T], \  \label{eq:es_power} 
\\
    & q_{\omega t} \leq  \w^\ese + x^\ese, \ \forall t \in [T], \label{eq:es_energy}
\\
    & q_{\omega t} - r^\text{dis}_{\omega t} \geq 0, \ \forall t \in [T], \label{eq:es_energy_lb}
\\
    & q_{\omega, t+1} = q_{\omega,t} + p^\text{chg}_{\omega t} \eta^\es - p^\text{dis}_{\omega t} / \eta^\es, \ \forall t \in [T],  \label{eq:soc_continuity}
\\
    &  { q_{\omega 1} = q_{\omega, T+1} }
    \label{eq:con_soc0},
\\
    & q_{\omega t},
p^\mathrm{chg}_{\omega t},
p^\mathrm{dis}_{\omega t}, r^\text{dis}_{\omega t} \geq 0, \ \forall t \in [T],
    \label{eq:con_es_nonneg}
    \end{align}
\end{subequations}
where $p^\es_{\omega t}\in\R^{n^\es}$ is the net output vector from $n^\es$ storage devices at time $t$ composed of charging $p^\chg_{\omega t}$ and discharging $p^\dis_{\omega t}$ decisions, $r^\dis_{\omega t}$ are storage-provided reserves, $q_{\omega t} \in \R^{n^\es}$ are the states of charge, and $\w^\esp, \w^\ese \in \R^{n^\es}$ are existing storage power and energy capacities. Eq. \eqref{eq:es_power}-\eqref{eq:soc_continuity} limit the total usage of storage, accounting for withheld capacity for reserves and storage dynamics following standard linear constraints with efficiency $\eta^\es$. {Eq. \eqref{eq:es_power} addresses charge-discharge exclusivity based on Lemma \ref{eq:lemma_storage} below}. Constraint \eqref{eq:con_soc0} {enforces the cyclic constraint described in Section \ref{sec:time_coupling}}. A system reserve margin $\gamma^\dd$ is applied to the total load $\overline{p}^d_{\omega t}$,
\begin{align}
\label{eq:reserves}
    \mathbf{1}^\top r_{\omega t}^\g + \mathbf{1}^\top r_{\omega t}^\text{dis} \geq \gamma^\dd \mathbf{1}^\top \overline{p}^\dd_{\omega t}, \quad \forall t\in[T],
\end{align}
\rev{and other reserve types (e.g., downward regulation, spinning and non-spinning reserves) can also be added as additional LP constraints within the mathematical framework.}

{
\begin{lemma}
\label{eq:lemma_storage}
 Eq. \eqref{eq:es_power},\eqref{eq:con_es_nonneg} are equivalent to the continuous relaxation of charge-discharge disjunctive constraints, namely:
 \begin{align*}
p^\mathrm{chg}_{\omega t i} &\leq (\overline{w}^\esp_i + x^\esp_i) z_{\omega t i},
\\
p^\mathrm{dis}_{\omega t i} + r^\mathrm{dis}_{\omega t i} &\leq (\overline{w}^\esp_i + x^\esp_i) (1 - z_{\omega t i}).
 \end{align*}
\end{lemma}
\begin{proof}
A binary variable $z_{\omega t i} \in \{0,1\}$ can enforce non-simultaneous charge-discharge with the two inequalities stated above. Now consider the continuous relaxation $z_{\omega t i} \in [0,1]$. First, adding both inequalities produces Eq. \eqref{eq:es_power}, so it is a valid inequality. Conversely, given variables $(p^\chg_{\omega t i},p^\dis_{\omega t i}, r^\dis_{\omega t i})$ satisfying Eq. \eqref{eq:es_power} and \eqref{eq:con_es_nonneg}, we can easily construct a $z_{\omega t i} := p^\chg_{\omega t i} / (\overline{w}^\esp_{i} + x^\esp_{i})$. Since $p^\chg_{\omega t i},p^\dis_{\omega t i}, r^\dis_{\omega t i} \geq 0$, and thus $p^\chg_{\omega t i} \leq (\overline{w}^\esp_{i} + x^\esp_{i})$, we have $z_{\omega t i} \in [0,1]$. The first inequality $p^\mathrm{chg}_{\omega t i} \leq (\overline{w}^\esp_i + x^\esp_i) z_{\omega t i} = p^\mathrm{chg}_{\omega t i}$ is true, by construction. Next, Eq. \eqref{eq:es_power} implies $p^\dis_{\omega t i} + r^\dis_{\omega t i} \leq (\overline{w}^\esp_i + x^\esp_i) - p^\chg_{\omega t i} = (\overline{w}^\esp_i + x^\esp_i)(1 - z_{\omega t i})$, which is the second inequality. This proves the continuously-relaxed disjunctive inequalities are equivalent to Eq. \eqref{eq:es_power},\eqref{eq:con_es_nonneg}.
\end{proof}

A normalized approximation error metric for storage complementarity is $\min(p^\dis_{\omega t i}, p^\chg_{\omega t i}) / (\overline{w}^\esp_{i} + x^\esp_{i})$, which is 0 when fully disjunctive. We quantify this error to have an average and median of within 5\%, in our numerical test case, justifying our relaxed storage formulation in the CEM context.
}

{
\subsubsection{Unit Commitment (UC)}
\label{sec:uc}
A subset of generators are subject to UC constraints, which are clustered based on zone and technology, following \cite{palmintier2013heterogeneous}. These UC variables are continuously-relaxed, following \cite{poncelet2020unit, pecci2025regularized}. We have
\begin{subequations}    
\begin{align}
\label{eq:uc_z_x} 
    &\Delta^\uc \odot z^{\uc} = A^\mathrm{uc}(\w^\g + x^\g),
\\
\label{eq:uc_u_z}
    & 0 \leq u_{\omega t}, s^\uc_{\omega t}, d^\uc_{\omega t} \leq z^{\uc}, \quad \forall t \in [T],
\\
\label{eq:uc_min_power}
    &A^\mathrm{uc} p^\g_{\omega t} \geq \mu^{\text{min-}\uc} \odot \Delta^\uc \odot u_{\omega t}, \quad \forall t \in [T],
\\
\label{eq:uc_max_power}
    &A^\mathrm{uc} (p^\g_{\omega t} + r^\g_{\omega t}) \leq  \Delta^\uc \odot u_{\omega t}, \quad \forall t \in [T],
\\
\label{eq:uc_dynamics} 
    &u_{\omega , t+1} = u_{\omega t}  + s^\uc_{\omega t}  - d^\uc_{\omega t}, \quad \forall t \in [T],
\\
\label{eq:uc_min_up}
    &u_{\omega t  \mathfrak{c}} \geq \sum\nolimits_{\tau = 
[t - m^{s}_\mathfrak{c}]_1
}^t s^\uc_{\omega \tau \mathfrak{c}} , \quad \forall \mathfrak{c}\in \mathcal{G}^\uc, t \in [T],
\\
\label{eq:uc_min_down}
    &z^\uc_\mathfrak{c} - u_{\omega t \mathfrak{c}} \geq \sum\nolimits_{\tau = [t - m^{d}_\mathfrak{c}]_1}^t   d^\uc_{\omega \tau \mathfrak{c}}, \forall \mathfrak{c}\in \mathcal{G}^\uc, t \in [T],
\\
\label{eq:uc_cyclic}
    &u_{\omega 1} = u_{\omega, T+1}.
\end{align}
\end{subequations}
where \eqref{eq:uc_z_x} computes continuously-relaxed unit counts $z^\uc {\in} \mathbbm{R}^{|\mathcal{G}^\uc|}$ with nominal nameplate capacities $\Delta^\uc {\in} \mathbbm{R}^{|\mathcal{G}^\uc|}$. Here, $A^\uc {\in} \{0,1\}^{|\mathcal{G}^\uc| \times n^\g}$ assigns the set of $n^\g$ generators to the set $\mathcal{G}^\uc$ of zone-technology clusters, e.g., coal, combined-cycle, combustion turbine, etc. (generators without UC constraints would simply have 0 entries in $A^\uc$). Eq. \eqref{eq:uc_u_z} bounds the clustered variables $u_{\omega t},s^\uc_{\omega t},d^\uc_{\omega t} \in \mathbbm{R}^{|\mathcal{G}^\uc|}$, i.e., commitment status, startup, and shutdown decisions. Eq. \eqref{eq:uc_min_power}-\eqref{eq:uc_max_power} enforce minimum and maximum power constraints, using parameters $\mu^{\text{min-}\uc} {\in} \mathbbm{R}^{|\mathcal{G}^\uc|}$. Eq. \eqref{eq:uc_dynamics} enforces the startup and shutdown dynamics, while Eq. \eqref{eq:uc_min_up}-\eqref{eq:uc_min_down} enforce minimum up and down time constraints, where $m^s_\mathfrak{c},m^d_\mathfrak{c}$ are the minimum up and down time parameters for cluster $\mathfrak{c}$, and $[\cdot]_1$ stands for $\max(1,\cdot)$. Eq. \eqref{eq:uc_cyclic} is the cyclic constraint for UC status.
}

\subsubsection{Long-Duration Storage} 
\label{eq:ldes}
For reservoir hydropower, we adapt the formulation of \cite{pecci2025regularized,genx_hydro_reservoir_module},
\begin{subequations}
\label{eq:hydro}
\begin{align}
\label{eq:hy_minpower}
    &p^\hy_{\omega t} \geq \mu^{\text{min-}\hy}_{\omega t} \odot \overline{w}^\hy, \quad \forall t \in [T],
\\
\label{eq:hy_dynamics}
    &q^\hy_{\omega, t+1} \leq q^\hy_{\omega t} + A^\hy (\overline{\rho}^\hy_{\omega t}\odot \overline{w}^\hy - p^\hy_{\omega t}), \quad \forall t \in [T],
\\
\label{eq:hy_reservoir_ub} 
    &0 \leq q^\hy_{\omega t} \leq \overline{x}^\mathrm{inv}, \quad \forall t \in [T+1],
\\
\label{eq:hy_inv_start}
    & q^\hy_{\omega 1} = x^\mathrm{inv}_\omega, \qquad
    q^\hy_{\omega, T+1} = x^\mathrm{inv}_{\rev{\omega + 1}},
    \qquad 
    \rev{x^\mathrm{inv}_{|\Omega|+1} = x^\mathrm{inv}_1},
\end{align}
\end{subequations}
where $p^\hy_{\omega t} \in \mathbbm{R}^{n^\hy}$ are nodal hydro generation, lower-bounded in \eqref{eq:hy_minpower} with min-generation parameters $\mu_{\omega t}^{\text{min-}\hy}$, \rev{and already upper-bounded by} \eqref{eq:gen_power_reserves} as part of $p^\g$. Here $A^\hy {\in} \{0,1\}^{|\mathcal{G}^\hy| \times n^\hy}$ assigns $n^\hy$ hydro resources to the set $\mathcal{G}^\hy$ of zone-aggregated reservoirs. Eq. \eqref{eq:hy_dynamics} models the dynamics of reservoir levels $q^\hy_{\omega t} \in \mathbbm{R}^{|\mathcal{G}^\hy|}$, where the inequality allows for water spillage. \rev{S}easonal inflows (represented with coefficients $\overline{\rho}^\hy_{\omega t}$) are netted with hydro generation, and then zonally clustered. Eq. \eqref{eq:hy_reservoir_ub} bounds the reservoir states. \rev{In \eqref{eq:hy_inv_start}, the first constraint copies week $\omega$'s starting storage state as variable $x^\mathrm{inv}_\omega$, and the second constraint equates week $\omega$'s ending state to the beginning of next week ($\omega{+}1$) via $x^\mathrm{inv}_{\omega+1}$. The last equation wraps the end of the year to the beginning.}

\edit{
\begin{lemma}[Year-long LDES trajectory]
    Problem \eqref{eq:cem} and Eq. \eqref{eq:hydro} enforce a full year of hydro scheduling.
\end{lemma}
\begin{proof}
Problem \eqref{eq:cem} enforces Eq. \eqref{eq:hy_inv_start} for each $\omega$. By substituting each $x^\mathrm{inv}_{\omega}$, it is apparent they altogether enforce
$$q^\hy_{1,1},\  \left[\text{\emph{Week 1}}\right], \ q^\hy_{1,T+1}=q^\hy_{2,1}, \  \left[\text{\emph{Week 2}}\right], \ q^\hy_{2,T+1}=q^\hy_{3,1}, ...,$$
and \eqref{eq:hy_dynamics} enforces internal storage dynamics within each week. Thus, full-year hydro scheduling is decided consistently.
\end{proof}}

\subsubsection{Cycle-based DC Power Flow} \label{sec:cycle_impfeedback}
DC power flow satisfies standard nodal power balance. Denoting nodal net power injections as $p^\tni_{\omega t}\in\R^N$, we have for all times $t\in[T]$,
\begin{subequations}
\label{eq:nodal_dc_kcl}
    \begin{align}
     & p^\tni_{\omega t} = A^\g p^\g_{\omega t}  + A^\es p_{\omega t}^\text{es} 
     \,
     {-}
     \,
     A^\dc p^\dc_{\omega t} -
     A^\text{d}(\overline{p}_{\omega t}^\text{d} - p^\text{sh}_{\omega t}),   \label{eq:nodal_balance}
\\
\label{eq:pni_pbr}
    & p^\tni_{\omega t} = A^\br p^\br_{\omega t},
    \end{align}
\end{subequations}
where $A^\g {\in} \{0,1\}^{N\times n^\g},\ 
A^\es {\in} \{0,1\}^{N\times n^\es},\ 
A^\text{d} {\in} \{0,1\}^{N\times {n^\dd}}$ are incidence matrices for generators, storage, and loads, respectively, and $p^\sh_{\omega t}\in\R^{n^d}$ is the vector of load shedding at time $t$. The vector $p^\br_{\omega t} \in \R^b$ of AC branch flows and $p^\dc_{\omega t} \in \R^\beta$ of HVDC line flows are constrained by ratings,
\begin{align}
\label{eq:hard_branch_limits}
-(\w^\br + x^\br )&\le p^\br_{\omega t} \leq \w^\br + x^\br , \ \forall t \in [T],
\\
-\w^\dc &\le p^\dc_{\omega t} \leq \w^\dc, \ \forall t \in [T], \label{eq:dc_line_limits}
\end{align}
where $ \overline{\omega}^\br \in\R^b, \overline{\omega}^\dc \in\R^\beta$ are the existing capacities for AC and HVDC branches, respectively. 

Recently, \cite{kocuk2016cycle} discovers a more computationally efficient way to express DC power flow using cycle bases. As background, a \emph{cycle} is a sequence of distinct vertices ${\nu}_1,\dots,{\nu}_k$ where each consecutive pair 
$({\nu}_i,{\nu}_{i+1})$ is connected by an edge and the last vertex reconnects to the first. When combining two cycles, their edge-incidence vectors are added modulo 2 (denoted $\oplus$), so edges appearing in both cancel out; this produces an \emph{even-degree subgraph} in which every vertex is incident to an even number of edges. The set of all even-degree subgraphs forms the graph's \emph{cycle space}, a vector space over the field $\mathbbm{F}_2=\{0,1\}$. A \emph{cycle basis} is a set of linearly independent cycles spanning the cycle space \cite{diestel2024graph}.

Given a power network, we can find a directed cycle basis matrix $D \in \{-1,0,1\}^{n^c \times b}$, where $n^c = b - N + 1$ is the cycle space's dimension \cite{diestel2024graph}, and each row of $D$ describes a cycle's incidence vector. Then Kirchhoff's Voltage Law (KVL), i.e., the difference of voltage angles across an edge should sum to zero over all edges in a cycle, can be written as
\begin{align}
\label{eq:cycle_basis}
    & \sum_{j \in \mathcal{J}_\kappa} D_{\kappa j}\cdot \chi_j({x}^\br_j)\cdot p^\br_{\omega t j} = 0, \quad \forall \kappa \in [n^c], \ t \in [T],
\end{align}
where $\chi_j({x}^\br_j)$ is the impedance of branch $j$. It is proved in \cite{kocuk2016cycle} that DC power flow is equivalent to the set of constraints \eqref{eq:nodal_dc_kcl} plus \eqref{eq:cycle_basis} over a cycle basis. We use this \emph{cycle-based DC power flow} extensively in our model.

\subsubsection{Impedance Feedback}
Importantly, \eqref{eq:cycle_basis} expresses the branch impedance $\chi_j(x^\br_j)$ as a function of capacity $x^\br_j$. We model this relationship as a continuous function,
\begin{align}
\label{eq:law_of_parallel_circuits}
&\chi_j({x}^\br_j) = {\chi_j^0  {\w^\br_j}} / ({\w^\br_j + {x}^\br_j}),
\end{align}
based on the law of parallel circuits \cite{hagspiel2014cost}, where $\chi^0_j$ is the original branch impedance prior to expansion. In general, our framework accommodates any continuous relationship $\chi_j(\cdot)$. We term this co-dependence of impedance and the capacity decision \emph{impedance feedback}. The division in \eqref{eq:law_of_parallel_circuits} makes the pair \eqref{eq:cycle_basis}-\eqref{eq:law_of_parallel_circuits} nonlinear in $x^\br$. Impedance feedback is considered in \cite{neumann2019heuristics}, which sequentially solves a full capacity expansion LP, but without a theoretical convergence justification.

\subsubsection{Transmission Contingencies}\label{sec:trans_cont}
We consider $n{-}1$ preventive transmission contingencies defined over the set $\mathcal{B}\subseteq[b]$, which comprises non-islanding branches in the network (also known as non-bridge edges). Note that an edge is a bridge if and only if it is not contained in any cycle; so we construct $\mathcal{B}$ as the set of edges that appear at least once in a cycle basis. Define the \emph{full} set of contingency indices as
\begin{align}
\label{eq:J_full}
    \mathcal{J}^\full = \{ (t,i,j) \in [T] \times [b] \times \mathcal{B}: \  i\neq j \},
\end{align}
where each contingency $(t,i,j)$ is indexed by a triplet of time $t$, the monitored branch $i$, and a different contingency-outaged branch $j$. The $n{-}1$ security constraints are expressed as follows similarly to \cite{mehrtash2020security}, now using slack variables $s^\cc$. For all $(t,i,j) \in \mathcal{J}^\full$, it must hold that,
\begin{subequations}
\label{eq:contingencies}
\begin{align}
\label{eq:post_ctg_flow}
    p^\brc_{\omega tij} &= p^{\br}_{\omega ti} + \Lambda_{ij}({x}^\br) p^\br_{\omega tj},
\\
\label{eq:lodf_neg}
    p^\brc_{\omega tij} &\geq -\eta^\cc (\overline{w}^\br_i + x^\br_i) -  s^\cc_{\omega tij} , 
\\
\label{eq:lodf_pos}
    p^\brc_{\omega tij} &\leq \eta^\cc(\overline{w}^\br_i + x^\br_i) +  s^\cc_{\omega tij}, 
\\
    s_{\omega tij}^\cc &\geq 0,
    \label{eq:ctg_slack}
\end{align}
\end{subequations}
where $p^\brc_{\omega tij}$ is the branch flow of $i$ under contingency $j$, $\Lambda_{ij}(x^\br)$ is the line outage distribution factor (LODF) for branch $i$ under contingency $j$, and $\eta^\cc \geq 1$ is the scalar multiple for the post-contingency rating. The LODF matrix $\Lambda \in \R^{b \times b}$ can be constructed using the power transfer distribution factor (PTDF) matrix $\Phi \in \R^{b \times (N-1)}$, following\cite{castelli2024improved,ronellenfitsch2016dual}:
\begin{subequations}
\label{eq:network_matrices}
\begin{align}
\label{eq:Phi_matrix}
        \Phi({x}^\br) &= B({x}^\br) A [A^\top B({x}^\br) A]^{-1},
\\
\label{eq:LODF_equation}
        \Lambda({x}^\br) &= \Phi({x}^\br) A^\top  [ I - \operatorname{diag}(\Phi({x}^\br)A^\top )  ]^{-1},
\end{align}
\end{subequations}
where $B(x^\br) = \operatorname{diag}(\chi({x}^\br))^{-1}$ is the diagonal matrix of branch susceptances. In \eqref{eq:Phi_matrix}-\eqref{eq:LODF_equation}, $A \in \R^{b \times (N-1)}$ is defined as $(A^\br)^\top$ with its slack bus column removed, and $\operatorname{diag}$ in \eqref{eq:LODF_equation} denotes keeping the diagonal part of the matrix while zeroing the rest. Due to the matrix inversion, constraints \eqref{eq:network_matrices} add further nonlinearity to the impedance feedback effect. For notational simplicity \eqref{eq:LODF_equation} has all branches, while \eqref{eq:J_full} only uses non-islanding branches as contingencies.

\subsubsection{Overall Operational Subproblem}\label{sec:oper_subprob}
The operational problem's objective function consists of generator variable costs plus penalties from load shedding and branch limit violations
\begin{align}
\label{eq:subproblem_obj}
    z_\omega^\top y_\omega {:=}  \sum_{t=1}^T  [(c^\g_{\omega t})^\top p_{\omega t}^\g
    { + (c^\uc)^\top s^\uc_{\omega t}}
    + c^{\sh} \mathbf{1}^\top p^{\sh}_{\omega t}
    + c^\vio  \mathbf{1}^\top s^\cc_{\omega t} ],
\end{align}
where $y_\omega$ is the vector of all operational decisions introduced above, {$c^\uc \in \mathbbm{R}^{|\mathcal{G}^\uc|}$ is the vector of startup costs}, $c^\sh$ and $c^\vio$ are scalar penalty coefficients, and 
$\mathbf{1}$ is a vector of ones with appropriate dimension so $\mathbf{1}^\top p^{\sh}_{\omega t}$ and $\mathbf{1}^\top s^\cc_{\omega t}$ denote the sum of all components of $p^{\sh}_{\omega t}$ and $s^\cc_{\omega t}$.

Putting everything together, we can now precisely define the operational subproblem and feasible region for scenario $\omega$ as
\begin{align}
\label{eq:subproblem_full}
h(x,\xi_\omega)
:&= 
\min\nolimits_{y_\omega \in \mathcal{Y}(x, \xi_\omega)} \ z_\omega^\top y_\omega,
\\
\mathcal{Y}(x,\, \xi_\omega) :&=
\left\{
    y_\omega  : \;\eqref{eq:gen_limits}-\eqref{eq:network_matrices}
\right\}. \label{eq:Y}
\end{align}

\section{Algorithms}

The capacity expansion model \eqref{eq:cem} together with the scenario subproblems \eqref{eq:subproblem_full}-\eqref{eq:Y} impose severe computational challenges due to the huge scale and nonlinearity. In particular, the scenario subproblems have DC power flow over large nodal networks, a large number of time intervals across scenarios, and a large number of transmission contingencies. Moreover, impedance feedback introduces a difficult nonlinearity. We propose several algorithmic approaches to deal with these challenges: 1) At the highest level, we propose a linear approximation of the overall nonlinear capacity expansion model with gradually tightened relaxations (see \ref{sec:alg:approx_subprob} and \ref{sec:alg:oracle}) and then use a novel fixed-point algorithm to correct the nonlinear impedance feedback effect (\ref{sec:alg:fixed_point}); 2) We adopt a modified level-bundle method to solve the linear expansion model {with an inexact oracle to capture contingencies} (\ref{sec:alg:bundle}); 3) For the linear operational subproblems, we introduce a fast algorithm for the cycle-based DCOPF (\ref{sec:alg:cycleOPF}). {These concepts and their relationships are previewed in  Fig. \ref{fig:flowchart}.}

\begin{figure}[H]
    \centering
    \includegraphics[width=1\linewidth]{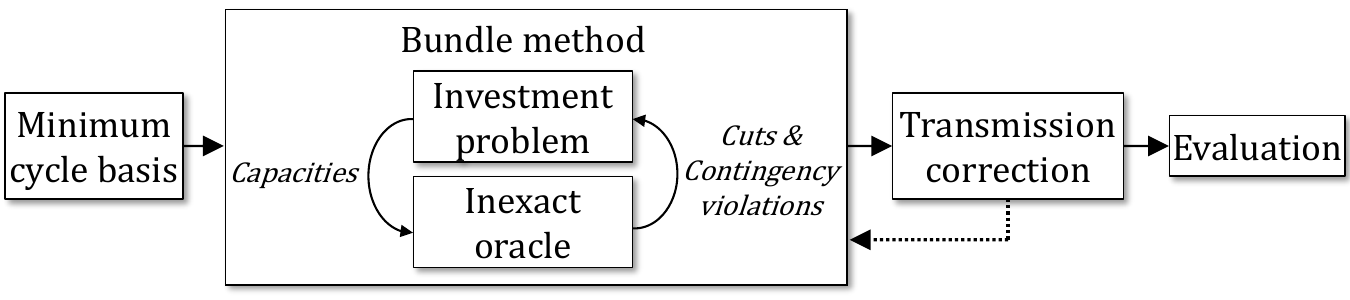}
    \caption{{Conceptual flowchart for the algorithmic components of CANOPI.}}
    \label{fig:flowchart}
\end{figure}

\subsection{Approximate operational subproblem}\label{sec:alg:approx_subprob}
To remove nonlinearity, we fix the variable $x^\br$ in \eqref{eq:cycle_basis}, \eqref{eq:law_of_parallel_circuits},  \eqref{eq:post_ctg_flow}, and \eqref{eq:network_matrices} as a parameter $\hat{x}^\br$, termed ``impedance-defining capacity''. For a fixed $\hat{x}^\br$, the values of $\chi(\hat{x}^\br)$, $B(\hat{x}^\br)$, and $\Phi(\hat{x}^\br)$ in \eqref{eq:network_matrices} can be pre-computed. Note that the $x^\br$ variable in \eqref{eq:hard_branch_limits} and
\eqref{eq:lodf_neg}-\eqref{eq:lodf_pos} is still treated as a first-stage variable, not as the fixed parameter $\hat{x}^\br$, for the purpose of generating cutting planes in the bundle method. {This fixed-impedance linearization yields a linear approximation that can be solved to desired optimality with valid lower and upper bounds, and which is verifiable via ex-post validation, unlike the nonlinear formulation which offers no such guarantees and remains computationally intractable.}

To improve tractability over the $O(b^2)$ possible contingency constraints, we introduce a relaxation of the operational feasibility sets $\mathcal{Y}$, by requiring the constraints \eqref{eq:contingencies} be satisfied only for a subset of contingencies $\mathcal{J}_\omega \subset \mathcal{J}^\full$ for each scenario $\omega$. We will later systematically tighten the relaxation. 
Combining the impedance-approximation and the contingency-relaxation, we define a revised operational feasibility set,
\begin{align}
\mathcal{Y}^r(\hat{x}^\br, x,\xi_\omega, \mathcal{J}_\omega) = \{ & y_\omega : \; \eqref{eq:gen_limits}-\eqref{eq:dc_line_limits} \text{ with } x, \xi_\omega,  
\\
    & \eqref{eq:cycle_basis}, \eqref{eq:law_of_parallel_circuits},
    \eqref{eq:post_ctg_flow},
    \eqref{eq:network_matrices} \text{ with } x^\br = \hat{x}^\br, \nonumber
\\
    & \hspace{-0.5cm}\eqref{eq:lodf_neg}-\eqref{eq:ctg_slack} \text{ with $x^\br$}, \forall (t,i,j)\in\mathcal{J}_\omega \}, \nonumber
\end{align}
which is a set of \emph{linear} constraints in $x$ for fixed $\hat{x}^\br$. This feasibility set also has complete recourse over $x$, i.e., $\mathcal{Y}^r$ is nonempty for any $x$, thanks to slack variables. Then the revised operational subproblem's optimal value function is
\begin{align}
\label{eq:subproblem}
h^r(\hat{x}^\br, x,\xi_\omega, \mathcal{J}_\omega)
:= 
\min_{y_\omega \in \mathcal{Y}^r(\hat{x}^\br, x,\xi_\omega, \mathcal{J}_\omega) } \ & z_\omega^\top y_\omega.
\end{align}

The linear approximation of CEM based on a particular $\hat{x}^\br$ (we {initialize} $\hat{x}^\br = 0$) can be expressed as a two-stage optimization, termed BUND (for bundle method), where each scenario $\omega$ considers the full contingency set $\cJ^\full$:
\begin{align}
(\bund) \ \min_{x \in \mathcal{X}}\ &c^\top x + \sum\nolimits_{\omega {\in \Omega}} h^r(\hat{x}^\br, x,\xi_\omega, \mathcal{J}^\full).\label{eq:bund}
\end{align}

In Section \ref{sec:alg:bundle}, we introduce a bundle method to obtain a highly accurate estimate of $\bund$'s optimal value. This is achieved by solving the subproblems \eqref{eq:subproblem} initially with $\mathcal{J}_\omega = \emptyset$ and systematically updating $\cJ_\omega$. Before this, we introduce an oracle for generating violated contingency constraints.

\subsection{Contingency constraint-generation oracle}\label{sec:alg:oracle}

Define the following oracle $\mathcal{O}$,
\begin{subequations}
\label{eq:oracle}
\begin{align}
    \mathcal{O}:\ &(\hat{x}^\br, x,\xi_\omega,\mathcal{J}_\omega) \mapsto (y^*_\omega, \theta^*_\omega, g^*_\omega, \sigma^\cc_\omega, \mathcal{J}'_\omega) \text{ s.t.}\notag
\\
\label{eq:O_minimizer}
    y^*_\omega &\in \argmin_{y_\omega \in \mathcal{Y}^r(\hat{x}^\br,x,\xi_\omega,\mathcal{J}_\omega)} z_\omega^\top y_\omega,
\\
\label{eq:O_subgradient}
    \theta^*_\omega &= z_\omega^\top y^*_\omega,
    \text{ and }
    g^*_\omega \in \partial_x h^r(\hat{x}^\br,x,\xi_\omega,\mathcal{J}_\omega),
\\
\label{eq:O_implied_slacks}
    \hat{s}^\cc_{\omega tij} &= \left[
       \left| p^\br_{\omega ti} + \Lambda_{ ij}(\hat{x}^\br) p^\br_{\omega tj} \right| -\eta^\cc  (\overline{w}^\br_i + x^\br_i)\right]^+,
\\
\label{eq:nu_c}
 \sigma^\cc_\omega &= c^\text{vio} \sum\nolimits_{(t,i,j) \in \mathcal{J}'_\omega} \hat{s}^\cc_{\omega tij},
\\
 \label{eq:Jnew}
   \mathcal{J}'_\omega &= \{(t,i,j)\in\mathcal{J}^\full \setminus \mathcal{J}_\omega : \hat{s}^\cc_{\omega tij} > 0 \},
\end{align}
\end{subequations}
where \eqref{eq:O_minimizer}-\eqref{eq:O_subgradient} find a minimizer $y^*_\omega$, the optimal value $\theta^*_\omega$, and a subgradient $g^*_\omega$ of the approximate operational subproblem $h^r(\hat{x}^\br, x, \xi_\omega, \cJ_\omega)$. The oracle also computes transmission slack in \eqref{eq:O_implied_slacks}, where $[\cdot]^+:=\max\{\cdot,0\}$, and returns the total contingency penalty $\sigma^\cc_\omega$ in \eqref{eq:nu_c} based on the index set $\mathcal{J}'_\omega$ of new violated contingencies in \eqref{eq:Jnew}.

\begin{proposition}
\label{prop:bounds} (Lower and upper bounds).
For any impedance-defining capacity $\hat{x}^\br$, scenario $\xi_\omega$, and a subset of contingencies $\mathcal{J}\subseteq\mathcal{J}^\full$, consider the following quantities computed at two capacity decisions $x, z\in\mathcal{X}$
\begin{subequations}
\begin{align}
\theta^\full &\leftarrow  h^r(\hat{x}^\br, x,\xi_\omega,\mathcal{J}^\full),
\\
(y^*,\theta^{*x},\cdot,\sigma^\cc,\cdot) &\leftarrow \mathcal{O}(\hat{x}^\br, x, \xi_\omega, \mathcal{J}),
\\
(\cdot,\theta^{*z},g^*,\cdot,\cdot) &\leftarrow \mathcal{O}(\hat{x}^\br, z, \xi_\omega, \mathcal{J}).
\end{align}
\end{subequations}
Then, $(\theta^{*z}, g^*)$ and $(\theta^{*x}, \sigma^\cc)$ provide valid lower and upper bounds on the $\mathcal{J}^\full$-optimal value $\theta^\full$ as
\begin{align}
\theta^{*z} + (g^*)^\top(x - z) &\leq \theta^\full
\leq \theta^{*x} + \sigma^\cc.
\end{align} 
\end{proposition}

\begin{proof}

First, $\theta^{*z} + (g^*)^\top(x - z) \leq h^r(\hat{x}^\br,x,\xi_\omega,\mathcal{J}) \leq \theta^\full$, where the first inequality is due to the convexity of $h^r$ in $x$ and the second inequality is due to the relaxation $\cJ\subseteq\cJ^\full$. For each of the previously ignored indices $(t,i,j)\in \mathcal{J}^\full\setminus \mathcal{J}$, the oracle constructs a contingency slack $\hat{s}^\cc$ that satisfies constraints \eqref{eq:contingencies}. Then the augmented operational solution $\hat{y} = (y^*, \hat{s}^\cc)$ is feasible for $\cY^r$ with $\cJ^\full$, and $\hat{y}$'s subproblem objective \eqref{eq:subproblem_obj} equals the relaxed objective $\theta^{*x}$ plus the new violation penalty $\sigma^\cc$. Thus, we have $\theta^\full \leq \theta^{*x} + \sigma^\cc$. 
\end{proof}

\subsection{Bundle method with interleaved constraint generation}
\label{sec:alg:bundle}
We develop a bundle-type method in Alg. \ref{alg:bundle} to solve $\bund$ \eqref{eq:bund}. It has the basic structure of a level-bundle method \cite{nesterov2018lectures} with two crucial differences. Each iteration $k$ builds cutting plane models $\hat{h}_{k\omega}(x)$ and $\hat{f}_k(x)$, which by Prop. \ref{prop:bounds} are lower approximations of operational objectives $h^r(\hat{x}^\br, x,\xi_\omega,\cJ^\full)$ and $\bund$'s overall objective \eqref{eq:bund}, respectively. This is achieved by solving, \emph{in parallel}, the linear approximate subproblems via the oracle $\cO$ in line \ref{alg:bundle:O}, obtaining cutting planes in line \ref{alg:bundle:cuts}, and aggregating in line \ref{alg:bundle:fk_hat}. Minimizing the lower approximation $\hat{f}_k$ in line \ref{alg:bundle:L} gives a lower bound $L_k$ of $\bund$, while by Prop. \ref{prop:bounds}, $f_k$ in lines \ref{alg:bundle:fk}-\ref{alg:bundle:U} gives an upper bound $U_k$. The algorithm terminates if $U_k$ and $L_k$ are sufficiently close. Otherwise, a level set of $\hat{f}_k$ is defined with a target level $\theta^{lev}_k$ as $\mathcal{L}(\hat{f}_k, \theta^{lev}_k) := \bigl\{ x \in \mathcal{X}: \  \hat{f}_k(x) \leq \theta^{lev}_k  \bigr\}$, where $\theta^{lev}_k$ is chosen as a convex combination of $U_k$ and $L_k$ in line \ref{alg:bundle:theta}. 

A crucial departure from the standard level-bundle method is in line \ref{alg:bundle:x}, where the next iterate $x_{k+1}$ is found as the analytic center of the level set $\mathcal{L}(\hat{f}_k, \theta^{lev}_k)$. In comparison, the standard level-bundle method projects $x_k$ to $\mathcal{L}(\hat{f}_k, \theta^{lev}_k)$ by solving a quadratic program, which is more computationally intensive \cite{pecci2025regularized}. Recall the analytic center of a convex set $\cZ$ is defined as $ac(\mathcal{Z}):= \argmax_{x\in\mathcal{Z}} F(x)$, where $F$ is a self-concordant barrier of $\cZ$ \cite{nesterov2018lectures}. This variant of the level-bundle method that leverages the analytic center cutting plane method (ACCPM) \cite{gondzio1996accpm} is proposed in \cite{zhang2025integrated}. 

We further improve upon the above level-bundle variant by integrating contingency generation in the process. Rather than fully solving each subproblem with $\cJ^\full$ before generating cuts for the capacity decision, the oracle $\cO$ returns newly identified contingency violations found from partial screening, which are added to the contingency list in line \ref{alg:bundle:add_ctg}. To our knowledge, this combination of adaptive or inexact oracles (based on systematic constraint tightening) with an analytic center bundle method has not been previously published.

\begin{algorithm}[t]
\caption{Bundle method with interleaved contingencies}
 \begin{algorithmic}[1]
 \label{alg:bundle}
 \renewcommand{\algorithmicrequire}{\textbf{Input:}}
 \renewcommand{\algorithmicensure}{\textbf{Output:}}
 \REQUIRE $\epsilon > 0$, $\alpha \in (0,1)$, {and initial $\hat{x}^\br$}. 
 \ENSURE $x^*$ and $y^*$.

\STATE Initialize bounds $L_0 \leftarrow 0, U_0 \leftarrow \infty$, and some $x_1 \in \mathcal{X}$.
\STATE Initialize models $\{\hat{h}_{0\omega} \leftarrow 0\}_\omega$ and sets $\{\mathcal{J}_{1\omega} \gets\emptyset \}_\omega$.

\FOR{$k=1,2,...$}

    \FOR{scenario $\omega \in \Omega$, in parallel}  \label{alg:bundle:forloop_start}

        \STATE $(y_{k\omega}, \theta_{k\omega}, g_{k\omega}, \sigma_{k\omega}, \mathcal{J}'_{k\omega}) \leftarrow\mathcal{O}({\hat{x}^\br}, x_k,\xi_\omega,\mathcal{J}_{k\omega})$. \label{alg:bundle:O}

        \STATE $\hat{h}_{k\omega}(x) \leftarrow \max\{ \hat{h}_{k-1,\omega}(x), \theta_{k\omega} + (g_{k\omega})^\top(x-x_k)   \}$.\label{alg:bundle:cuts}

        \STATE Add constraints \label{alg:J_new} $\mathcal{J}_{k+1,\omega} \leftarrow \mathcal{J}_{k\omega} \cup \mathcal{J}'_{k\omega}$. \label{alg:bundle:add_ctg}
    \ENDFOR \label{alg:bundle:forloop_end}

    \STATE $\hat{f}_k(x)\leftarrow c^\top x + \sum_\omega \hat{h}_{k\omega}(x)$.\label{alg:bundle:fk_hat}

    \STATE ${f}_k \leftarrow c^\top x_k + \sum_\omega [ \theta_{k \omega} + \sigma_{k\omega}]$.\label{alg:bundle:fk}

    \STATE $L_k \leftarrow \min_{x\in\mathcal{X}} \ \hat{f}_k(x)$ \label{alg:bundle:L}, and $U_k \leftarrow\min\{ U_{k-1}, f_k\}$.\label{alg:bundle:U}

    \STATE \algorithmicif\ $U_k = f_k$ \algorithmicthen\ $x^* \leftarrow x_k, \text{ and } y^* \leftarrow \{y_{k\omega}\}_\omega.$

    \STATE \algorithmicif\ $(U_k - L_k)/U_k < \epsilon$ \algorithmicthen\
       \textbf{return} $x^*, y^*$. \textbf{break}.

    \STATE $\theta_k^{lev}\leftarrow L_k + \alpha(U_k - L_k)$. \label{alg:bundle:theta}

    \STATE $x_{k+1} \leftarrow ac (
    \{ x \in \mathcal{X}: \  \hat{f}_k(x) \leq \theta_k^{lev}  \}
     )$. \label{alg:bundle:x}

\ENDFOR
\end{algorithmic}
\end{algorithm}

\begin{proposition}
\label{prop:bundle}
    Alg. \ref{alg:bundle} terminates in finite iterations and returns an $\epsilon$-optimal solution of the $\bund$ problem \eqref{eq:bund}.
\end{proposition}
\begin{proof}
    At iteration $k$, if the gap $(U_k-L_k)/U_k$ has reached the desired tolerance $\epsilon$, then the algorithm terminates with an $\epsilon$-optimal solution because the lower and upper bounds are valid by Prop. \ref{prop:bounds}. Otherwise, there are two possible iteration types, discernible in lines \ref{alg:bundle:cuts}-\ref{alg:bundle:add_ctg}. \textbf{Type~I}: At least one subproblem $\omega$ either (a) adds a cut that locally improves $\hat{h}_{k,\omega}(x_k) > \hat{h}_{k-1,\omega}(x_k)$, or (b) generates new constraints with $\mathcal{J}'_{k\omega} \neq \emptyset$. There are a finite number of possible subsets $\mathcal{J}_\omega \subseteq \mathcal{J}^\full$, and each $\mathcal{J}_\omega$-parameterized LP \eqref{eq:subproblem} has a finite number of \emph{faces} since each face is defined by a set of active constraints. Recall that faces can range in dimension, from vertices (0), edges (1), polygons (2), ... , up to facets with $\dim(\cY^r(...)) - 1$. Then, each subgradient cut in line \ref{alg:bundle:cuts} contains at least one of these faces. So the number of Type~I iterations, where some cut adds at least one new face to the cutting plane model, must be finite. In a \textbf{Type~II} iteration: for all scenarios $\omega$, we have both (a') $\hat{h}_{k\omega}(x_{k}) = \hat{h}_{k-1, \omega}(x_{k})$, and (b') $\sigma_{k\omega} = 0$. Thus,
    \begin{align*}
        f_{k}
    &=
        c^\top x_{k} + \sum\nolimits_{\omega \in \Omega} \theta_{k\omega}
    \leq 
        c^\top x_{k} + \sum\nolimits_{\omega \in \Omega} \hat{h}_{k\omega}(x_{k})
\\
    &=
        c^\top x_{k} + \sum\nolimits_{\omega \in \Omega} \hat{h}_{k-1,\omega}(x_{k})
    =
        \hat{f}_{k-1}(x_{k})
    \leq \theta^{lev}_{k-1},
    \end{align*}
    where the first equality follows from $f$'s definition in line \ref{alg:bundle:fk} and assumption (b'), the first inequality uses the max operator in line \ref{alg:bundle:cuts}, the second equality applies assumption (a'), the third equality uses $\hat{f}$'s definition in line \ref{alg:bundle:fk_hat}, and finally the last inequality follows from the membership of $x_{k}$ in $\mathcal{L}(\hat{f}_{k-1}, \theta^{lev}_{k-1})$ from line \ref{alg:bundle:x}. This fact means $U_{k} \leq f_{k} \leq L_{k-1} + \alpha(U_{k-1} - L_{k-1})$. Further, the nondecreasing lower bound $L_{k} \geq L_{k-1}$ implies $U_{k} - L_{k} \leq \alpha (U_{k-1} - L_{k-1})$, i.e., the gap has improved geometrically.    Then $\epsilon$ convergence is guaranteed after
    $K > \log\left(\frac{1}{\epsilon} \cdot \frac{U_1 
- L_1}{f^*} \right)   /  \log(\frac{1}{\alpha})$
    iterations of Type II. Thus Alg. \ref{alg:bundle} converges.
\end{proof}

Unlike \cite{zhang2025integrated}'s method and convergence proof, which must evaluate oracles  at additional unstabilized Benders iterates (requiring further solve times), Prop. \ref{prop:bundle} proves the convergence of our hybrid level-ACCPM method where \emph{only} the stabilized points $x_k$ are evaluated with subproblem oracles. {We also remark that \cite{pecci2025regularized} omits a convergence proof.}

\subsection{Transmission correction for impedance feedback}
\label{sec:alg:fixed_point}

Alg. \ref{alg:bundle} solves $\bund$ \eqref{eq:bund}  to get $(x^*,y^*)$. We will fix the non-transmission decisions $(x^\nbr_*, y^\nbr_*)$ from $(x^*, y^*)$. Then we wish to make the branch capacities $x^\br$ and the impedance-defining parameters $\hat{x}^\br$ consistent, i.e., $x^\br = \hat{x}^\br$, in order to satisfy the impedance feedback constraints \eqref{eq:cycle_basis}, \eqref{eq:law_of_parallel_circuits}, \eqref{eq:post_ctg_flow}, \eqref{eq:network_matrices}. We do this with an iterative transmission correction process (CORR), illustrated in Fig. \ref{fig:RTEP_iterative}. First, we set the impedance-defining parameter $\hat{x}^\br$ to the $x^\brs$ solution. Then, we re-optimize branch capacities $x^\br$ to minimize capacity and contingency costs. We call this re-optimization the \emph{restricted transmission expansion problem} (RTEP), which produces branch capacities $\tilde{x}^\br$. Then we update $\hat{x}^\br$ to $\tilde{x}^\br$, and we continue to re-solve RTEP until $\tilde{x}^\br$ and $\hat{x}^\br$ converge. {This overall BUND-CORR procedure can be repeated multiple times, by feeding CORR's final output $\tilde{x}^\br$ back into BUND's initial impedance-defining capacities $\overline{w}^\br$ (while keeping track of the branch investments' sunk cost, which is needed because the BUND objective \eqref{eq:bund} only models the investment cost $(c^\br)^\top x^\br$ for new upgrades relative to $\overline{w}^\br$).}

\begin{figure}[ht]
    \centering
    \includegraphics[width=0.90\linewidth]{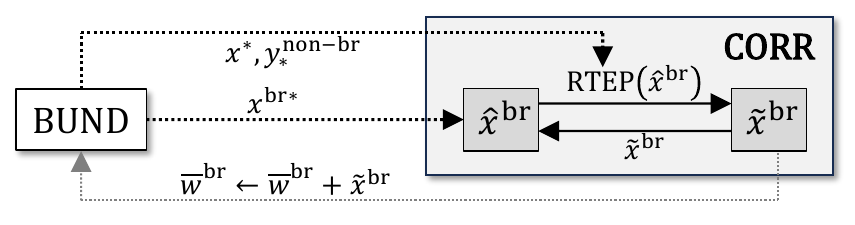}
    \vspace{-10pt}
    \caption{Iterative procedure CORR to re-optimize a consistent $\hat{x}^\br = \tilde{x}^\br$. {The inputs $(x^*, y^\nbr_*)$} are fixed across all iterations of RTEP, while $\hat{x}^\br$ is updated iteratively.}
    \label{fig:RTEP_iterative}
\end{figure}

Given a fixed non-transmission solution $(x^\nbr_*, y^\nbr_*)$ {and branch solution $x^{\br*}$ (which is taken as an element-wise lower-bound target)}, RTEP is parametrized by a general $\hat{x}^\br$,
\begin{align}
\label{eq:restricted_tep}
\text{(RTEP}) &  \ 
 \min_{x^\br, \{y^\br_\omega\}_\omega} \ (c^\br)^\top x^\br + c^\vio \sum\nolimits_{\omega {\in \Omega}} \mathbf{1}^\top s^\cc_\omega  
\\
\text{s.t.} & \quad {x^{\br*}} \leq x^\br \leq \overline{x}^\br, \nonumber
\\
\multicolumn{2}{c}{
    $({y}_{*\omega}^\nbr, y_\omega^\br) \in 
\mathcal{Y}^r\!\left(\hat{x}^\br,(x^\nbr_*,x^\br),\xi_\omega,\cJ^\full\right), \ \forall \omega {\in\Omega},$
} \nonumber
\end{align}
where the objective preserves relevant cost terms from \eqref{eq:bund}, and $y^\br_\omega=(p^\br_\omega, p^\brc_\omega, s^\cc_\omega)$ are re-calculated power flow variables. {Note that RTEP is always \emph{anchored} to BUND's transmission decision $x^{\br*}$ as a lower bound.}
{In the overall CORR framework (Fig. \ref{fig:RTEP_iterative}), the linearized capacity expansion model is followed by re-optimizing transmission decisions and updating impedance-related parameters in a self-consistent way. In contrast, in \cite{neumann2019heuristics} the full-scale capacity expansion LP is followed by merely a simple impedance update (without re-optimizing transmission). In other words, \cite{neumann2019heuristics}'s approach relies purely on the full-scale LP (which is expensive to solve) to drive transmission decisions, while the CANOPI framework provides two opportunities to improve transmission: both BUND and CORR. This algorithmic choice is enabled by computational speed:} RTEP can be solved by Alg. \ref{alg:E_function}, which only requires lightweight algebraic operations.

\begin{proposition}
    Alg. \ref{alg:E_function} {gives an optimal solution to} RTEP \eqref{eq:restricted_tep}.
\end{proposition}

\begin{proof}
Alg. \ref{alg:E_function} only needs to consider $p^\tnis_\omega \in y^\nbr_\omega$ since the other components of $y^\nbr_\omega$ remain feasible to the non-power-flow constraints \eqref{eq:gen_limits}-\eqref{eq:pni_pbr}. Then, given $\{p^\tnis_\omega\}_\omega$ and $\hat{x}^\br$ as inputs, the power flows $\hat{p}^\br$ are uniquely determined by the standard PTDF mapping in line \ref{alg:E:apply_ptdf_formulation}, where the $[2:N]$ indices omit the slack bus. This PTDF mapping is equivalent to the DC power flow constraints \eqref{eq:nodal_dc_kcl} and \eqref{eq:cycle_basis}. Post-contingency power flows $\hat{p}^\brc$ are similarly determined in line \ref{alg:E:p_brc}.

At this point, both pre- and post-contingency power flows are determined, and for RTEP it only remains to optimize the tradeoff between costly transmission investments $x^\br$ versus violations $\{s^\cc_\omega\}_\omega$. Lines \ref{alg:E:deltas_b} and \ref{alg:E:max_base} calculate a lower bound $x^\text{br-lb}_i$ to satisfy base-case feasibility constraints \eqref{eq:hard_branch_limits} across all scenarios and time periods. For contingencies, line \ref{alg:E:deltas_c} calculates $\hat{\delta}^\cc_{\omega tij}$ to identify the minimal contingency slack that satisfies constraints \eqref{eq:contingencies} as $s^\cc_{\omega t i j} = [\hat{\delta}^\cc_{\omega tij} - \eta^\cc x^\br_i]^+$, which is a function of $x^\br_i$. So the $y^\br_\omega$ variables can be projected out from RTEP, leaving an equivalent problem \eqref{eq:restricted_tep_separable} involving only $x^\br$, which is now \emph{separable} across branches $i$, {by considering branch $i$'s possible contingencies $\mathcal{J}^\br := \Omega \times [T] \times \mathcal{B}$}:
\begin{subequations}
\label{eq:restricted_tep_separable}
\begin{align}
\min_{{ x^\br_i }}\ \  & c^\br_i  x^\br_i +  c^\vio
\sum\nolimits_{
(\omega,t,j) {\in \mathcal{J}^\br}
} 
\left[\hat{\delta}^\cc_{\omega tij} {-} \eta^\cc x^\br_i\right]^+,
\\
{\text{s.t.}} \ \ &
{x^\br_i \in [\max (x^{\br*}_i, x^\text{br-lb}_i),\    \overline{x}^\br_i] }
,
\end{align}
\end{subequations}
where the dependence on $\hat{x}^\br$ is embedded in $x^\text{br-lb}_i$ and $\hat{\delta}^\cc_{\omega tij}$. The subdifferential of $[\cdot]^+$ equals: $\{0\}$ when the argument is negative, $[0, 1]$ when $0$, and $\{1\}$ when positive. So the unconstrained optimality condition for \eqref{eq:restricted_tep_separable} is
\begin{align}
    \frac{c^\br_i}{\eta^\cc  c^\vio}
        \in
         \sum_{(\omega,t,j) {\in \mathcal{J}^\br} } \left[
         {
                \mathbbm{1}( 
                \frac{
                    {\hat{\delta}^\cc_{\omega t i j}}
                }
                {\eta^\cc}
                {>} 
                x^\br_i ),
                \ 
                \mathbbm{1}( 
                \frac{
                    {\hat{\delta}^\cc_{\omega t i j}}
                }
                {\eta^\cc}
                {\geq}   x^\br_i)
            }
            \right], \label{eq:corr_stationarity}
\end{align}
where the $\mathbbm{1}$ indicators count the number of terms in \eqref{eq:restricted_tep_separable}'s summation which have nonzero derivative. The expression in \eqref{eq:corr_stationarity} is a step function of $x^\br_i$ that decrements with a vertical segment at every breakpoint in the array $v_i = \{\hat{\delta}^\cc_{\omega tij} / \eta^\cc \}_{\omega t j}$. Thus, starting from the right limit where \eqref{eq:corr_stationarity}'s expression is 0 at $x^\br_i \to \infty$, assigning $x^\br_i$ to the $r_i := c^\br_i / (\eta^\cc c^\vio)$ largest breakpoint reaches the correct number of steps and satisfies \ref{eq:corr_stationarity}'s unconstrained optimality condition. This calculation is performed by lines \ref{alg:E:c_ratios}-\ref{alg:E:uncon_opt} in Alg. \ref{alg:E_function}. Finally, line \ref{alg:E:con_opt} projects the unconstrained optimal solution $x^\text{opt}$ onto the feasible interval $[{\max(x^{\br*}_i, } x^\text{br-lb}_i), \overline{x}^\br_i]$ from \eqref{eq:restricted_tep_separable}. This solves RTEP.
\end{proof}


\begin{algorithm}[t]
\caption{Function $E$ for restricted transmission expansion}
 \begin{algorithmic}[1]
 \label{alg:E_function}
 \REQUIRE initial $\hat{x}^\br$ and nodal net injections $\{ p^\tnis_\omega \}_\omega$.
 \ENSURE  updated $\tilde{x}^\br \in \R^b$.
  
  \FOR {$i \in [b]$}

      \FOR{$\omega \in \Omega, \ t\in[T]$}
           \STATE
           \label{alg:E:apply_ptdf_formulation}
           $\hat{p}^\br_{\omega t i} \gets \Phi_i(\hat{x}^\br) p_{\omega t,[2:N]}^\tnis.$ 
             \STATE
         \label{alg:E:p_brc}
         $\hat{p}^\brc_{\omega tij} \gets 
                    \hat{p}^\br_{\omega t i} + \Lambda_{ij}(\hat{x}^\br) \hat{p}^\br_{\omega t j}, \forall j \in \mathcal{B}.$ 
                   
           \STATE
           \label{alg:E:deltas_b}
           $\hat{\delta}_{\omega t i}^\text{base} \gets 
         \left[
             | \hat{p}^\br_{\omega t i}   | - \overline{w}^\br_i
         \right]^+.$

           \STATE
           \label{alg:E:deltas_c}
           $\hat{\delta}^\cc_{\omega tij} \gets 
                    \left[
                        | \hat{p}^\brc_{\omega tij} | -\eta^\cc \overline{w}^\br_i
                    \right]^+, \forall j \in \mathcal{B}.$ 
    \ENDFOR

     \STATE
     \label{alg:E:max_base}
     $x^\text{br-lb}_i \gets \max_{\omega \in \Omega, t\in[T]} \{ \hat{\delta}_{\omega t i}^\text{base} \}$

     \STATE
     \label{alg:E:c_ratios}
     $r_i \gets  \left\lceil c_i^\br  / (\eta^\cc c^\text{vio}) \right\rceil$, and $v_i \gets $ array $\{\hat{\delta}^\cc_{\omega tij} / \eta^\cc \}_{\omega t j}.$

     \STATE
     \label{alg:E:uncon_opt}
     $x^\text{opt}_i \gets$ the $r_i$-th largest value in array $v_i.$

     \STATE
     \label{alg:E:con_opt}
     $\tilde{x}^\br_i \gets \min\left\{
            \max\{
                 {x^{\br*}_i,}
                 x^\text{br-lb}_i,
                 x^\text{opt}_i
               \}, \
            \overline{x}^\br_i
            \right\}.$
  \ENDFOR
  
 \RETURN $\tilde{x}^\br$.
 \end{algorithmic}
\end{algorithm}

Alg. \ref{alg:E_function} takes $\hat{x}^{\br}$ as input and computes an optimal solution $\tilde{x}^{\br}$ of RTEP. This defines a function, which we denote as $\tilde{x}^{\br}=E(\hat{x}^{\br})$. Using this notation, the CORR procedure in Fig. \ref{fig:RTEP_iterative}
describes a \emph{fixed-point iteration}: $\hat{x}^{\br}_{k+1}=E(\hat{x}^{\br}_k)$. We now show that $E$ indeed has a fixed point.

\begin{proposition} \label{prop:E_fixed_point}
A fixed point $\hat{x}^\br = E(\hat{x}^\br)$ exists.
\end{proposition}

\begin{proof}
We will apply Brouwer’s Fixed-Point Theorem, which states that every continuous function from a nonempty compact convex subset of a finite-dimensional Euclidean space to itself has a fixed point \cite{smart1980fixed}. First, the function $E$ is a mapping from $[0, \overline{x}^\br]$ to itself, since $[x^\text{br-lb}_i, \overline{x}^\br_i] \subseteq [0, \overline{x}^\br_i]$. Moreover, $[0,\overline{x}^\br]$ is convex, compact, and nonempty. Next, we show that $E$ is continuous in its inputs $\hat{x}^\br$, by describing it as a composition of continuous functions. The underlying $\chi_j(\cdot)$ function is assumed to be continuous. Matrix inversion to calculate PTDF is continuous over the space of full-rank square matrices. Similarly, matrix inversion to calculate LODF is continuous for branches in $\mathcal{B}$, since their ``self-PTDF'' terms in $\operatorname{diag}(\Phi({x}^\br)A^\top )$ from \eqref{eq:LODF_equation} are not 1. Further, choosing the $r$-th largest element in $v\in\R^m$, i.e., the order statistic operation, is continuous since it can be expressed as a composition over a finite set of max/min operations based on only $m$ and $r$, namely
$
\max_{S\subseteq[m]:\ |S|=r}  \left\{ \min_{j\in S} v_j \right\}
$.
Each $r$-sized subset contains $r$ elements that are greater than or equal to the inner minimum, $\min_{j\in S} v_j$. Maximizing over such subsets yields the $r$-th largest value in $v$. Thus the order statistic operator is continuous. The remaining compositions involve max, min, and affine operators (including the selection on element indices), which are continuous. Thus $E$ is continuous, and it has a fixed point by Brouwer's fixed-point theorem.
\end{proof}

\subsection{Fast algorithm for cycle-based DCOPF}\label{sec:alg:cycleOPF}
While the $\bund$ method introduced in Section \ref{sec:alg:bundle} provides a tractable algorithm, it still requires multiple calls to the subproblem oracle. Significant computational complexity is created by the KVL constraints, especially for large networks. To efficiently formulate KVL, \cite{kocuk2016cycle} proposes constraining cycle flows; the authors use LU factorization to calculate a cycle basis. Previous work \cite{horsch2018linear} reports significant computational speedups when formulating the linearized power flow as decomposed on a spanning tree and a cycle basis, when compared to the angle formulation; \cite{horsch2018linear} uses a \emph{fundamental cycle basis} based on a spanning tree. No power systems paper seems to have used \emph{minim\rev{um}} cycle bases for DC power flows. 


Meanwhile, the graph algorithms literature has studied the minim\rev{um} cycle basis (mcb) problem \cite{liebchen2005greedy, kavitha2008algorithm, mehlhorn2009minimum}, which identifies a graph's cycle basis with a \emph{minim\rev{um}} total number of edges. Since sparsity of the coefficient matrix affects solver speed, we consider applying \rev{mcb} to DCOPF, rather than arbitrary cycle bases in prior {power systems literature \cite{kocuk2016cycle,horsch2018linear}}. Polynomial-time mcb algorithms exist, but they rely on specialized graph routines, e.g., repeated shortest path solves{, which become prohibitively slow on large networks.}

{We therefore make two contributions. First, we introduce the idea of mcb to the power systems literature. Second, to overcome limitations of existing graph theory algorithms, we develop a direct integer programming (IP) algorithm.}

Let $C_{\kappa j}\in\{0,1\}$ denote the incidence of edge $j$ in cycle $\kappa$. To improve cycle $C_{\hat{\kappa}}$, we solve an IP that searches over linear combinations of cycles including $C_{\hat{\kappa}}$ (to preserve linear independence) and minimizes the number of edges:
\begin{subequations}
\label{eq:milp}
\begin{align}
\min_{w,\zeta,v} \quad &\sum\nolimits_{j\in[b]} v_j \\
    \text{s.t. } \  & \sum\nolimits_{\kappa \in [n^c] } C_{\kappa j} \cdot w_\kappa = 2 \zeta_j + v_j, \ \forall j\in[b], \label{eq:milp:xor}
\\
    & w \in \{0,1\}^{n^c}, \quad w_{\hat{\kappa}} = 1, \quad \zeta \in \mathbbm{Z}^{b}. \label{eq:milp:integral}
\end{align}
\end{subequations}

\begin{proposition}
The optimal solution $v^*$ of \eqref{eq:milp} is a shortest simple cycle linearly independent of $\{C_\kappa : \kappa \neq \hat{\kappa}\}$.
\end{proposition}

\begin{proof}
Given binary weights $w$ on the cycles, minimizing over $\zeta,v$ produces the mod-2 sum as $v$. So the feasible set for $v$ is 
$V_{\hat{\kappa}} := \{ v : v = \bigoplus\nolimits_{\kappa} w_\kappa C_\kappa , \; w_{\hat{\kappa}}=1 \}$, i.e., all linear combinations of cycles that include $C_{\hat{\kappa}}$. Since $C_{\hat{\kappa}}$ is independent of the other cycles, $v^*$ inherits this independence. At this point, $v^*$ is guaranteed to be an even-degree subgraph. It remains to verify that $v^*$ is indeed a simple cycle, i.e., connected with unique vertices. Being in the cycle space, $v^*$ decomposes into disjoint simple cycles (by separating disconnected components and splitting vertices of degree $>2$ as necessary). Thus we may write $v^* = F_1 \oplus \cdots \oplus F_\mu$, where each $F_\ell$ is a simple cycle. Expanding each $F_\ell$ onto the original basis $C$ gives $v^* = \bigoplus_{\ell=1}^\mu   \left( \bigoplus_{\kappa} m_\kappa^{F_\ell} C_\kappa \right)
    = \bigoplus_{\kappa} \left( \sum_{\ell=1}^\mu m_\kappa^{F_\ell} \bmod 2 \right) C_\kappa$. 
The coefficients must match, including $w_{\hat{\kappa}}^* = 1$, so there is at least one $m_{\hat{\kappa}}^{F_{\ell^*}}=1$. Hence $F_{\ell^*} \in V_{\hat{\kappa}}$ is feasible for \eqref{eq:milp}. So by optimality, $\|v^*\|_1 \leq \|F_{\ell^*}\|_1$. Since the $\{F_\ell\}_\ell$ are edge-disjoint, we have $\|v^*\|_1 = \sum_{\ell=1}^\mu \|F_\ell\|_1 \ge \|F_{\ell^*}\|_1$. It follows that $\mu=1$. Thus $v^*=F_1$ is a shortest simple cycle in $V_{\hat{\kappa}}$.
\end{proof}

\begin{algorithm}[ht]
\caption{Algorithm for efficient minim\rev{um} cycle basis}
 \begin{algorithmic}[1]
 \label{alg:mcb}
 \REQUIRE initial undirected cycle basis $C^0 \in \{0,1\}^{n^c \times b}$.
 \ENSURE  a minim\rev{um} cycle basis: undirected $C$, directed $D$.
 
\STATE copy $C \gets C^0$, initialize $D = \mathbf{0}_{n^c \times b}$
 
  \FOR {$\hat{\kappa} = 1, 2, ..., n^c$}
  
  \STATE $C_{\hat{\kappa}} \gets v^*$ from solving \eqref{eq:milp} on index $\hat{\kappa}$.

  \STATE Traverse cycle $\hat{\kappa}$'s nodes $(\nu_1, ..., \nu_{n_{\hat{\kappa}}})$, with $\nu_{n_{\hat{\kappa}}} = \nu_1$. \label{alg:traversal}
  
  \FOR{$i \in [n_{\hat{\kappa}} - 1]$}

  \STATE Identify branch $\varepsilon$ s.t. $\{\nu_i, \nu_{i+1}\} = \{i^{fr}_\varepsilon, i^{to}_\varepsilon\}$.
  \STATE $D_{\hat{\kappa}, \varepsilon} \gets 1$ if $\nu_i = i^{fr}_\varepsilon$, else $-1$.
  
  \ENDFOR
  \label{eq:cycle_traversal_end}
  
  \ENDFOR
 \RETURN $C, D.$
 \end{algorithmic}
\end{algorithm}

\begin{proposition}
    Alg. \ref{alg:mcb} produces a minim\rev{um} cycle basis.
\end{proposition}
\begin{proof}
At each iteration $\hat{\kappa}$, \eqref{eq:milp} selects the shortest cycle that is linearly independent of the other cycles. This replacement is optimal among all feasible cycles. By induction, $C$ remains a valid cycle basis, and every cycle in the final basis is the shortest possible given the others. No other basis achieves a smaller total cardinality, and the algorithm yields a minim\rev{um} cycle basis. Then, the cycle traversal in lines \eqref{alg:traversal}-\eqref{eq:cycle_traversal_end} assigns consistent orientations to produce a directed basis $D$.
\end{proof}

Our method follows the overall basis exchange {strategy} described in \cite{berger2009minimum}. Our key novelty is using an IP {algorithm} to perform each exchange, rather than a bespoke graph procedure. Although without a polynomial complexity guarantee, the IP approach improves practical implementation and performance.

\section{Numerical Results}

\subsection{Network data and calibration}

We extract a geographically accurate topology $(A^\br, A^\dc)$ for the US Western Interconnection with {PyPSA-Earth} \cite{parzen2023pypsa}, and we apply \cite{birchfield2016grid}'s engineering parameters to estimate initial branch capacities $\overline{w}^\br$ and impedances $\chi^0$ based on distances and voltage levels. Generators from EIA-860 are mapped to \rev{nodal} locations. This approach produces a more geographically-realistic network than purely synthetic grid data. \rev{We focus} on the bulk transmission system \rev{and} include branches at 230kV and above, following \cite{shawhan2014does}. \rev{Hydro} and UC clustering zones, $\mathcal{G}^\hy$ and $\mathcal{G}^\uc$, are based on \cite{pecci2025regularized}'s six EPA-IPM-based Western zones. \rev{UC parameters are from \cite{palmintier2013heterogeneous}.}

Next, we calibrate the initial estimated network to satisfy zero load shed and zero branch violation during a peak load day, following \cite{xu2020us}'s calibration approach. The aim is to produce a realistic approximation of the status quo grid. We perform this calibration by iteratively solving a small optimization problem within an impedance feedback procedure similar to Fig. \ref{fig:RTEP_iterative}, producing a mutually consistent pair of $\overline{w}^\br$ and $\chi^0$. 

We represent operations with 52 weekly-horizon, hourly-resolution scenarios. While we use a historical weather year of 2023, our framework accommodates operational scenarios covering multiple weather years. Historical zonal load time series are mapped to nodes based on zip code populations, following \cite{birchfield2016grid}, to produce $\overline{p}^\dd$. Wind and solar availability factors $a^\g_{\omega t}$ are derived from NOAA Rapid Refresh reanalysis data. {Hydropower parameters and inflow data are adapted from the dataset in \cite{pecci2025regularized}}. Generation capital costs from \cite{nrel2024atb} are annualized over 20 years to give $c$, and operating costs $c^\g_\omega$ use EIA estimates. For investment limits $\overline{x}$: land restrictions and ordinances from \cite{lopez2023impact} constrain wind and solar installation, and storage and geothermal are assumed to be available up to 1GW at each node with voltage of 345kV or higher. In this study, we restrict battery duration to 4 hours, i.e., $x^\ese {=} 4 x^\esp$. We model a policy goal of 80\% carbon-free generation, i.e., $\overline{x}^\tem {=} 0.2 \sum_{\omega t} \overline{p}^\dd_{\omega t}$. We assume a load shedding cost of $c^\sh {=} \$$10,000 / MWh, and a contingency violation penalty of $c^\vio {=} \$$2,000 / MWh, which are representative of typical values used by grid operators. We choose $\alpha {=} 0.3$ for the bundle method, following theoretical justification in \cite{lemarechal1995new}.

\subsection{Problem size and computational resources}
The above method creates a Western Interconnection network with 1,493 nodes, 3 HVDC lines, and 1,919 AC branches (of which 1,728 are lines and 1,542 are non-islanding branches). We model $|\Omega|{=}52$ scenarios of $T{=}168$ week-long, hourly-resolution operations. In total, this CEM instance has {78.7} million non-contingency variables and {138} million non-contingency constraints, along with 20 billion total contingency constraints and their associated slack variables.

We demonstrate the computational performance and solution accuracy impact from this paper's proposed algorithmic components. We implement algorithms in Julia (v1.\rev{12.6}) \cite{bezanson2017julia}. Linear programs are written in JuMP (v1.25.0) \cite{lubin2023jump} and solved with Gurobi (v\rev{13.0.2}) \rev{barrier method} \cite{gurobi2025}. Computation is performed on a single AMD EPYC 9474F node with {96 CPU cores and 376GB of RAM}, on MIT's Engaging Cluster.

\subsection{Impact of high-fidelity grid modeling}
\label{sec:bundle_experiments}

We test the $\bund$ algorithm under a range of grid physics representations, and we consistently evaluate all solutions on the original CEM objective function in \eqref{eq:cem} with {$n{-}1$ transmission security-constrained DCOPF}, labeled as \emph{``Total cost''} \rev{(in million USD per year)}. By construction, \eqref{eq:cem}'s objective uses \rev{a consistent} $x^\br$ for both ratings and impedance calculation, \rev{so} it incorporates impedance feedback. This evaluation is tractable to compute since $x$ is fixed, in contrast to CEM.

{Table \ref{tab:bundle_results} reports results from one major iteration of BUND-CORR (for results about repeating multiple BUND-CORR feedback iterations, please see Section \ref{sec:multi_bund_corr}).} In Table \ref{tab:bundle_results}, each column represents a level of grid fidelity captured in the BUND subproblems: {``Net. Flow'' means network flow which ignores KVL}, ``DC power flow'' ignores contingencies, ``DC-0.7'' adopts {PyPSA}'s $30\%$ branch-derating heuristic, and ``SC-DC'' is our full security-constrained model with ${n{-}1}$ contingencies. The BUND rows report the iteration counts, solution times (``Minutes''), and peak memory (``{RAM} GB'') required for BUND to converge to $\epsilon {=}1\%$ optimality gap. A majority of time is spent on operational subproblems (``$\mathcal{O}$ minutes''). Here the minim\rev{um} cycle basis formulation is used for all DC methods. We compare the final lower and upper bounds ($L_k$ and $U_k$, in {million} USD per year), and optimal total upgrades of {short-duration storage} and branches (in GW).

\begin{table}[t]
\centering
\caption{Performance and Impact of Contingencies and Grid Physics.}
\vspace{-5pt}
\label{tab:bundle_results}
\setlength{\tabcolsep}{3pt}
\begin{tabularx}{\columnwidth}{ Y r r r r r }
\toprule
\multicolumn{1}{c}{ } & \multicolumn{1}{c}{} & \multicolumn{3}{c}{{\emph{Contingencies in BUND? \textbf{No}} }  } & \multicolumn{1}{c}{{\textbf{\emph{Yes}}}}
\\
\cmidrule{3-5}
\cmidrule(lr){6-6}
\textbf{Method} & \textbf{Metric} & \textbf{Net. Flow} & \textbf{DC} & \textbf{DC-0.7} & \textbf{SC-DC} \\
\midrule
& Iterations & \rev{26} & \rev{56} & \rev{98} & 1\rev{04}
\\
& Minutes & \rev{31.2} & 2\rev{07.2} & \rev{357.0} & \rev{474.3}
\\
& $\mathcal{O}$ minutes & \rev{30.8} & 20\rev{5.2} & \rev{347.1} & \rev{450.2} \\
& {RAM (GB)} & \rev{337} & \rev{357} & \rev{350} & \rev{354}
\\
BUND & \cellcolor{gray!20} $L_k$ cost {(\$M)}
&\cellcolor{gray!20} \textbf{\${16,69\rev{0}}} 
&\cellcolor{gray!20} \textbf{\${16,\rev{891}}} 
&\cellcolor{gray!20} \textbf{\${17,5\rev{44}}} 
&\cellcolor{gray!20} \textbf{\${17,4\rev{03}}}
\\
&\cellcolor{gray!20} $U_k$ cost {(\$M)}
&\cellcolor{gray!20} \textbf{\${16,8\rev{51}}}
&\cellcolor{gray!20} \textbf{\${17,057}}
&\cellcolor{gray!20} \textbf{\${17,7\rev{21}}}
&\cellcolor{gray!20} \textbf{\${17,\rev{575}}}
\\
& Storage (GW) & 4.\rev{7} & \rev{4.5} & 5.\rev{5} & 4.\rev{8}
\\
& Branch (GW) & 3\rev{8.4} & 5\rev{5.2} & 1\rev{25.2} & 1\rev{11.0}
\\
\hline
& \cellcolor{gray!20}\textbf{Total cost} {(\$M)}
& \cellcolor{gray!20}\textbf{\${2\rev{46,504}}}
& \cellcolor{gray!20}\textbf{\${1\rev{49,233}}}
& \cellcolor{gray!20}\textbf{\${4\rev{1,760}}}
& \cellcolor{gray!20}\textbf{\${17,6\rev{16}}}
\\
{Eval. of \eqref{eq:cem}} & {{Cost diff. (\%)}}
& {{1\rev{299.4}\%}}
& {{7\rev{47.2}\%}}
& {{1\rev{37.1}\%}}
& --
\\
{(BUND)}
\normalsize & Shed (GWh) & 1\rev{8,708.5} & 9,\rev{327.6} & \rev{533.7} &
    \rev{3.2}
\\
& Viol. (GWh) & \rev{21,147.0} & 19,\rev{376.5} & 9,\rev{287.3} & \rev{21.5}
\\
\hline
& Iterations & \rev{98} & \rev{44} & \rev{92} & \rev{41} \\
CORR & Minutes & 2.\rev{1} & \rev{0.9} & \rev{1.9} & {0.9} \\
& Branch (GW) & 44\rev{5.1} & 30\rev{6.5} & \rev{199.4} & 1\rev{32.1} \\
\hline
& \cellcolor{gray!20}\textbf{Total cost} {(\$M)}
& \cellcolor{gray!20}\textbf{\$\rev{{25,042}}}
& \cellcolor{gray!20}\textbf{\${2\rev{5,146}}}
& \cellcolor{gray!20}\textbf{\${1\rev{7,985}}}
& \cellcolor{gray!20}\textbf{\${17,61\rev{2}}} \\
{Eval. of \eqref{eq:cem}} & {{Cost diff. (\%)}}
& \rev{{{42.2}\%}}
& \rev{{{42.8}\%}}
& {{2.\rev{1}\%}}
& -- \\
{(BUND}
& Shed (GWh)
& \rev{227.5}
& \rev{350.4}
& 2.\rev{3}
& 2.\rev{2} \\
{+CORR)} & Viol. (GWh)
& \rev{1,834.8}
& 1,\rev{628.3}
& \rev{12.5}
& \rev{2.0} \\
\bottomrule
\end{tabularx}
\end{table}

The contingency-aware bundle method \rev{(SC-DC)} converges {in ${\sim}$\rev{8}} hours, indicating tractability even at large scale. \rev{During its 104 iterations, the algorithm identifies and adds 2.1 million contingency constraints. At the final optimal solution, 733 thousand contingency constraints are binding (about 84 per operating hour, vs. about 5 binding base-case constraints per hour). The contingency-caused ratio of 94\%, or $84/(84+5)$, is consistent with Table \ref{tab:caiso_congestion}.}

\rev{We note the total number of contingency constraints considered is substantially smaller than the full possible set. In fact, the ability to use intelligent screening to selectively handle only the most impactful contingency constraints, while avoiding ``false negatives'' of missed constraints, is at the conceptual core of CANOPI's algorithm design. }

\subsubsection{{{Pre-CORR}}} Grid physics representation significantly impacts the optimal level of {branch upgrades}, suggest\rev{ing} that spatially-coarse expansion models under-invest in technologies with high locational value. Realistic evaluation (``BUND Evaluation'') reveals costs up to 14x higher for investment solutions produced by coarse models ({comparing \$\rev{246.5}B vs. \$17.6B}), driven by load shedding and branch violations. Ignoring contingencies \rev{can underestimate} system costs. 

\subsubsection{{{Post-CORR}}}
Subsequent remedial transmission upgrades (``CORR'') help reduce but cannot eliminate gaps: the solutions from network flow and DC remain \rev{42.2\%} and \rev{42.8\%} costlier {during evaluation than SC-DC (reported in ``Cost difference'' for ``BUND+CORR''). After CORR, the ``DC-0.7'' solution still has a nontrivial cost delta of \$\rev{17,985}-\rev{17,612} = \$\rev{373} million per year (2.\rev{1}\%) versus SC-DC, along with higher branch violations (\rev{12.5} vs. \rev{2.0}).} {Fig. \ref{fig:shadow_price_plot} plots top contingency constraints' shadow prices (averaged over hours of the year), which provide a summary metric of congestion severity; the results show that SC-DC can better mitigate contingency constraints than DC-0.7. Table \ref{tab:correlations} compares the Pearson correlations of nodal investments between different models. Storage and branch correlations remain below 0.8, and these differences in investment decisions cause DC-0.7's performance gap. Thus, CORR is helpful but} it is not a full substitute for planning with endogenous contingencies. 

\begin{figure}[ht]
    \centering
    \includegraphics[width=1\linewidth]{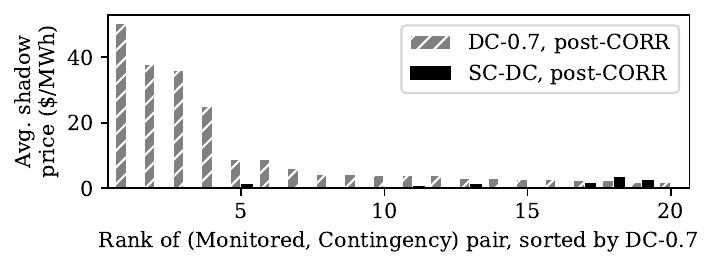}
    \vspace{-25pt}
    \caption{{Average shadow prices for the top 20 transmission contingency constraints from the post-CORR DC-0.7 solution, versus from SC-DC.}}
    \label{fig:shadow_price_plot}
\end{figure}
\vspace{-5pt}

\vspace{-5pt}
\begin{table}[ht]
\centering
\caption{Nodal investment correlations vs. SC-DC (pre \& post-CORR)}
\vspace{-5pt}
\label{tab:correlations}
\setlength{\tabcolsep}{3pt}
\begin{tabularx}{\columnwidth}{ Y  r r r r r r}
\toprule
\textbf{Method} & \textbf{Wind} & \textbf{Solar} & \textbf{Geoth.} & \textbf{Storage} & \textbf{Branch (pre)} & \textbf{Branch (post)} \\
\hline
Net. flow & 0.9\rev{6} & 0.6\rev{5} & 0.\rev{55} & 0.\rev{79} & 0.4\rev{2} & 0.45 
\\
DC & 0.96 & 0.\rev{80} & 0.\rev{73} & 0.\rev{78} & 0.5\rev{2} & 0.6\rev{3}
\\
DC-0.7 & 0.9\rev{9} & 0.9\rev{1} & 0.9\rev{1} & 0.7\rev{6} & 0.\rev{61} & 0.7\rev{9}
\\
\bottomrule
\end{tabularx}
\end{table}

\subsubsection{{The value of CORR and SC-DC methods}} 
The CORR method, \rev{one of this paper's contributions,} significantly improves cost and reliability, while incurring low computational burden (all solved within \rev{$\sim$2} minutes); it has relevant utility as a fast tool to comparatively evaluate capacity expansion models. At a \rev{comparable} computational cost to DC-0.7, the SC-DC method is inherently a more accurate model. Even \rev{if} a 2.\rev{1}\% cost delta is acceptable, \rev{calculating} such a \rev{bound} in the first place requires solving SC-DC. \rev{Further, Fig. \ref{fig:security_factor_rev2}'s grid search shows that varying the security factor to values other than 0.7 never closes the optimality gap.} Therefore, the \rev{CANOPI-enabled} SC-DC model provides valuable information beyond what is obtainable with a derating heuristic.

\begin{figure}[ht]
    \centering
    \includegraphics[width=1\linewidth]{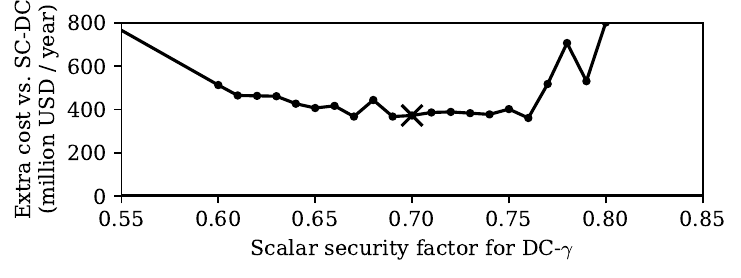}
    \vspace{-25pt}
    \caption{\rev{The extra costs of DC-$\gamma$ relative to SC-DC (post-CORR) for a range of scalar security factors $\gamma$ (DC-0.7 is marked with ``x'').}}
    \label{fig:security_factor_rev2}
\end{figure}
\vspace{-10pt}

\subsection{Impact of minim\rev{um} cycle basis}

{Operational subproblems take above 90\% of the solve times, in Table \ref{tab:bundle_results}. We assess the ability of mcb to speed up DCOPF subproblems.} First, we compare Alg. \ref{alg:mcb}'s performance with alternative calculation methods. In Table~\ref{tab:cycle_basis_calc}, a fundamental cycle basis is calculated using the spanning tree based on traversal starting from each node, and the average and minimal results are shown (where the entire search time is required to find the minimal {fundamental} basis). {The specialized mcb algorithm from \cite{kavitha2008algorithm} is implemented by \cite{hagberg2008networkx}.}

\begin{table}[ht]
    \centering
    \caption{Cycle basis methods: computation time and cycle lengths.}
    \vspace{-5pt}
    \label{tab:cycle_basis_calc}
    \begin{tabularx}{\columnwidth}{c l r r r}
\toprule
 {\textbf{Literature}} & \textbf{Method for cb} & \textbf{Seconds} & \textbf{Total}
& \textbf{Longest} \\
    \midrule
       & Fundamental {\cite{horsch2018linear}} & 1.4  & 9,640 & 236 
       \\
       {Power systems} &Fundamental (best)  & 16.8  & 7,434 & 191 
       \\
       & LU factorization {\cite{kocuk2016cycle}} & 0.8 & 8,854 & 150
       \\
       \hline
       {Graph theory} & mcb \cite{kavitha2008algorithm}\cite{hagberg2008networkx} & 1,294.0 & 2,755 & 33
\\
\hline
\multirow{2}{*}{{This paper}} & mcb Alg. \ref{alg:mcb}, HiGHS & 94.1 & 2,755 & 33
       \\
& \textbf{mcb Alg. \ref{alg:mcb}, Gurobi} & \textbf{23.4} & \textbf{2,755} & \textbf{33}
       \\
\bottomrule
    \end{tabularx}
\end{table}

In Table \ref{tab:cycle_basis_calc}, generic cycle bases can have total cardinality $\|C\|_1$ (``Total'') that is 3.5x \rev{larger} than that of a minim\rev{um} cycle basis (9640 / 2755), and have a largest cycle length of 7.2x \rev{longer than} mcb (236 / 33). On the other hand, {existing mcb algorithms are significantly slower than} non-minim\rev{um} bases. In contrast, our proposed Alg. \ref{alg:mcb} achieves the same sparsity while preserving a tractable calculation speed{,} achieving a 55x speedup with Gurobi (1294 / 23), and 14x with open-source solver HiGHS (1294 / 94).

In Table \ref{tab:cb_kvl_table}, {we test the impact of different KVL formulations}. The multi-period DC optimal power flow operational subproblems (168-hour horizon) are solved on \rev{40} combinations of  $x_k$ and $\{\mathcal{J_\omega}\}_\omega$ in order to provide input samples of different capacities and contingencies. For each sample, the solve time is taken as the max over $\omega$ (solved in parallel), and the statistics of the difference percentages (vs. mcb) are summarized in Table \ref{tab:cb_kvl_table}. {Using minim\rev{um} cycle basis (mcb) consistently outperforms all other KVL forms.} In particular, improved sparsity from using mcb results in a significant {average} speedup by \rev{25}\%-\rev{39}\% compared to generic cycle bases (``Mean diff.''), and by \rev{17}\% compared to the angle formulation {(with voltage $\theta$ angles)}. Alg. \ref{alg:mcb} improves the numerical sparsity of the DC power flow formulation, and can be a drop-in replacement when using a cycle flow formulation since the upfront basis calculation time is negligible. 


\vspace{-5pt}
\begin{table}[ht]
\centering
\caption{Impact of KVL formulation on subproblem solve times}
\vspace{-5pt}
\label{tab:cb_kvl_table}
\setlength{\tabcolsep}{3pt}
\begin{tabularx}{\columnwidth}{ Y   r r r }
\toprule
\textbf{KVL form} & \textbf{Mean diff. (\%)} &  \textbf{Median diff. (\%)}   & \textbf{\% of sample $>$ mcb} \\
\hline
PTDF           &  2\rev{89}\% &2\rev{84}\% & 100\% 
\\
$\theta$ angles &  \rev{17}\% &1\rev{6}\% & 9\rev{8}\%
\\
 cb (fund.) & \rev{25}\% & \rev{24}\% & 100\% \\
 cb (LU)  & \rev{39}\% & \rev{39}\% & \rev{100}\% \\
\bottomrule
\end{tabularx}
\end{table}
\vspace{-10pt}


\subsection{{Attributing} speedups from algorithm design choices}
\label{sec:speedup_decomp}

We quantify the contributions of CANOPI's algorithm design choices in speeding up {the BUND model with SC-DC. We focus on three main factors: (1) subproblem speedup, (2) inexact oracles with interleaving, and (3) analytic center in the level set problem. The first two techniques are novelties introduced by CANOPI, and the latter is explored in \cite{zhang2025integrated,pecci2025regularized} (while CANOPI is the first to apply it to the nodal SC-DC setting). Also, we note that CANOPI's combination of impedance approximation and CORR enables tractable solution of the overall algorithm, although we leave this structural design choice out of the following quantitative analysis.

{
Table \ref{tab:speedup_decomp} attributes the hours saved and percentages of total (``Saved \%'') to the different algorithm aspects. The ``Pre-CANOPI'' and ``Post-CANOPI'' columns report experimental measurements, except for values marked with ``$\sim$'' that are estimated based on the following analysis.
}

\begin{table}[ht]
\centering
\caption{Decomposition of BUND speedup (assumes \rev{104} iterations)}
\vspace{-5pt}
\label{tab:speedup_decomp}
\setlength{\tabcolsep}{3pt}
\begin{tabularx}{\columnwidth}{r r r r r r Y}
\toprule
\textbf{Algorithm aspect} & \textbf{Pre} & \textbf{Post} & \textbf{Hrs. saved} & \textbf{Saved \%} & \textbf{Innovation} \\
\hline
Subproblem (min.)  & $\sim$5.\rev{07} & 4.\rev{33} & $\sim$\rev{3.83} & $\sim$7.\rev{2}\% & mcb vs. $\theta$\\
Subproblem repeats & {5} & {1} & $\sim$3\rev{2.56} & $\sim$\rev{60.9}\% & interleaving  \\
\hline
Master problem (min.) & 1\rev{0.05} & 0.\rev{23} & $\sim$\rev{17.02} & $\sim$\rev{31.9}\% & $ac$ vs. QP\\
\hline
Total time (hours)  & $\sim$\rev{61.31} & \rev{7.91} & $\sim$\rev{53.41} & 100.0\% & CANOPI\\
\bottomrule
\end{tabularx}
\end{table}

First, using mcb speeds up the time per subproblem: \rev{achieving} 4.\rev{33} minutes on average (\rev{450} minutes / 1\rev{04} iterations; Table \ref{tab:bundle_results}), \rev{compared to the 17\%-slower voltage angle baseline} (Table \ref{tab:cb_kvl_table}) \rev{estimated at} 5.\rev{07} minutes. Further, interleaving reduces the number of subproblem re-solves. Empirically, SC-DCOPF evaluations require \rev{4}$\sim$7 iterations for contingency convergence (e.g., when calculating Table \ref{tab:bundle_results}); we \rev{estimate} 5 subproblem re-solves per first-stage iterate $x_k$ in the non-interleaved counterfactual. We assume total BUND iterations remain at \rev{104} (Table \ref{tab:bundle_results}).

The mcb and interleaving effects compound with each other. The subproblem time ($a {=} 4.\rev{33}$) multiplied by iterations ($b {=} 1$) gives the total time spent on subproblems ($a\cdot b$). Contrast this with the counterfactual subproblem time ($a'{=}5.\rev{07}$) and iterations ($b'{=}5$) in the absence of CANOPI techniques. We explain the contribution of each effect using the algebraic identity $a'b' - ab = (a'-a)\frac{(b'+b)}{2} + (b'-b)\frac{(a'+a)}{2}$. The first term attributes $(5.\rev{07} - 4.\rev{33})\frac{(5+1)}{2}\cdot \frac{1\rev{04}}{60} = \rev{3.83}$ hours saved by mcb's reduction of subproblem times. The second term attributes $(5-1)\frac{(5.\rev{07}+4.\rev{33})}{2}\cdot \frac{1\rev{04}}{60} = 3\rev{2.56}$ hours saved from interleaving's reduction of subproblem repeats. 

Second,} ``Master problem time'' refers to Alg. \ref{alg:bundle}'s lines \eqref{alg:bundle:fk_hat}-\eqref{alg:bundle:x}, the majority of which is spent on the master lower-bound problem in line \eqref{alg:bundle:L} and the analytic center problem in line \eqref{alg:bundle:x} to obtain the next iterate. Using the \emph{analytic center} approach rather than the traditional level method's quadratic projection objective leads to a significant speedup factor {(1\rev{0.05} to 0.\rev{23} minutes)}, which is even larger in our novel \emph{nodal} context than previously reported for zonal models \cite{pecci2025regularized}. {Finally, the ``Pre'' total hours is estimated as $(5.\rev{07}\times 5 + 1\rev{0.05}) \cdot \frac{1\rev{04}}{60}$.}

\vspace{-10pt}

{
\subsection{Comparison to direct solution of the extensive-form model}
\label{sec:direct}

To further contextualize the computational performance, we use Gurobi to directly solve the extensive form \rev{of the} BUND problem \eqref{eq:bund}, using a high-memory compute node with 1.5TB of RAM and 96 cores. A truly extensive-form model including several billion contingency constraints is computationally intractable. \rev{To replicate the SC-DC setting, we first solve without contingencies}, i.e., $\mathcal{J}_{\omega}{=}\emptyset$, \rev{and then iteratively screen and add contingency constraints} after each extensive-form solution. \rev{This is solved using 4 major extensive-form solves, totaling 1,322} minutes (\rev{22} hours), using \rev{677GB} of memory. Thus, using the Algorithm \ref{alg:bundle} decomposition to solve BUND with \rev{SC-DC} provides an \rev{overall 2.8x} speedup (\rev{1322} min. vs. \rev{474} min.) and \rev{1.9}x memory improvement (\rev{677GB} vs. \rev{354GB}).

\rev{Standard decomposition by itself (Table \ref{tab:speedup_decomp}'s ``Pre-CANOPI'') may be insufficient to outperform the extensive-form's solve time. CANOPI's methodological components, i.e., interleaving along with mcb and nodal-resolution modified master problem, enable the overall speedup.}

}

{
\subsection{Convergence of multiple BUND-CORR iterations}
\label{sec:multi_bund_corr}


We explore the convergence behavior of multiple BUND-CORR iterations. Table \ref{tab:bund_corr_iter} shows that the impedance feedback effect after major iteration 1 causes the evaluated costs (both pre-CORR 17,6\rev{16} and post-CORR 17,61\rev{2}) to exceed BUND's fixed-impedance upper bound $U_k$ of 17,\rev{575}. By major iteration 2\rev{,} both pre- and post-CORR costs lie squarely within BUND's bounds: $\{$17,5\rev{10}, 17,4\rev{95}$\} \subseteq $ [17,3\rev{57}, 17,5\rev{27}]. This suggests a kind of agreement between the two methods, implying limited improvement opportunities from further BUND-CORR repeats. \rev{T}he first post-CORR cost is only \rev{0.66}\% away from the final (17,\rev{495} vs. 17,\rev{612}). Thus, Table \ref{tab:bund_corr_iter} suggests that CANOPI converges to reasonable accuracy within only $\sim$\rev{2} major iterations of BUND-CORR. \rev{This result helps} validate the usage of one BUND-CORR iteration in Section \ref{sec:bundle_experiments}.

\begin{table}[ht]
\centering
\caption{Convergence behavior of BUND and CORR iterations}
\vspace{-5pt}
\label{tab:bund_corr_iter}
\setlength{\tabcolsep}{3pt}
\begin{tabular}{c r r }
\toprule
\textbf{Major BUND + CORR iteration} & \textbf{1} & \textbf{2} \\
\hline
BUND iterations & \rev{104} & \rev{72} \\
BUND minutes & \rev{474.3} & 289.\rev{8} \\
CORR minutes & 0.9 & 0.\rev{9}  \\
Total minutes change (\%) & -- & $-$\rev{39}\% 
\\
\hline
BUND new branches (GW) & 1\rev{11.0} & 3\rev{6.0}  \\
CORR new branches (GW) & 1\rev{32.1} & \rev{52.3}  \\
\hline
BUND $L_k$ cost (\$M) & \quad \$17,4\rev{03} & \quad \$17,3\rev{57}  \\
BUND $U_k$ cost (\$M) & \$17,\rev{575} & \$17,5\rev{27}  \\
\hline
Evaluation of \eqref{eq:cem} (BUND) (\$M) & \$17,6\rev{16} & \$17,51\rev{0}  \\
Evaluation of \eqref{eq:cem} (BUND+CORR) (\$M) & \$17,61\rev{2} & \$17,4\rev{95}  \\
\hline
\rev{Are evaluations within bounds $[L_k,U_k]$?} & \rev{No} & \rev{Yes}  \\
\rev{Final} cost delta vs.\ \rev{prior} iteration (\%) & -- & $-$0.\rev{66}\% \\
\bottomrule
\end{tabular}
\end{table}
}

{
\section{Extension to discrete branch investments}
\label{sec:extension}
The CANOPI framework may be extensible to \emph{endogenous} discrete branch investments, by adapting several mathematical techniques, though challenges remain. A promising starting point is \cite{mehrtash2020security}, which formulates discrete transmission expansion using shift factors from the full candidate network, \rev{representing} unbuilt lines via flow cancellation pairs. The underlying full topology's fixed nature is compatible with BUND's pre-computation assumptions. CANOPI's mcb reformulation could replace the base-case power flow constraints. However, each contingency \rev{may} still require its own full set of variables and constraints, \rev{versus} the sparsity of pre-computed LODFs. That said, \rev{contingency interleaving forms a majority of Table \ref{tab:speedup_decomp}'s speedup, and this benefit} would still apply. \rev{A}fter decomposition, subproblems remain LPs \rev{and fit} within Alg.\ \ref{alg:bundle}.
\par
For discrete decisions, \cite{pecci2025regularized}'s two-phase strategy first solves a continuously-relaxed CEM via level-ACCPM, then warm-starts an MILP with the generated cutting-plane model. Valid inequalities such as the cycle-based cuts in \cite{kocuk2016cycle} could strengthen and accelerate the MILP phase. In \cite{pecci2025regularized}'s numerical study, the continuous first phase dominates computation (about 7 out of 8 hours, or 88\% of total time). In addressing the continuous problem, our paper already tackles a computational bulk of the overall TEP problem. We leave detailed development and testing \rev{of this direction} to future work.
}

\section{Conclusion}

This paper introduces CANOPI, a comprehensive modeling and algorithmic framework for integrated nodal capacity expansion with endogenous transmission contingencies and detailed hourly operations {including representations of unit commitment and long-duration storage}. To our knowledge, this is the first work to solve such problems on a realistic scale. To achieve this, we develop a series of methodological contributions. First, we formulate a model that embeds impedance feedback effects. To ensure tractability, we construct a linearized approximation, combined with an algebraic fixed-point correction procedure that avoids repeatedly re-optimizing the full problem as impedances change. Second, we design a specialized bundle algorithm that unites analytic-center stabilization with adaptive contingency constraint generation. The algorithm design of \emph{interleaving} the iterations of constraint generation and of the bundle method avoids full convergence on contingency constraints at every bundle iteration, especially early on with poor quality investment solutions. Third, we present an IP routine to compute minim\rev{um} cycle bases, which produces sparser KVL constraints and reduces the solve times of operational subproblems.

These contributions connect disparate groups of research literature. We bridge macro-energy systems (focusing on accurate time domain representation) with transmission planning (including nodal power flows and contingency constraints). Our realistic-geography grid dataset enables comparisons with zonal models derived from real networks. Theoretically, Alg. \ref{alg:bundle}'s convergence proof reinforces connections between the distinct literature on the level method, ACCPM, and inexact oracles. The practical minim\rev{um} cycle basis algorithm, and its application to DC power flow, connects the graph theory, math programming, and power systems perspectives.

\rev{Our} numerical \rev{experiments} demonstrate the importance of nodal \emph{and} contingency-aware grid physics (Table \ref{tab:bundle_results}), \rev{and} the speedup contribution of each CANOPI algorithmic innovation (Table \ref{tab:speedup_decomp}). {We numerically demonstrate the convergence of the overall impedance correction framework (Section \ref{sec:multi_bund_corr}).} CANOPI lowers the computational barrier for researchers and practitioners to utilize more accurate nodal planning tools, translating to improved economic and reliability outcomes.

\bibliographystyle{IEEEtran}
\bibliography{ref}

\newpage

\vfill

\end{document}